\documentclass[11pt]{article}
\usepackage{graphicx}
\usepackage{amssymb}
\usepackage{epstopdf}
\usepackage{enumerate}
\usepackage{amsmath,amsfonts,amsthm,amssymb,verbatim,graphicx,subfigure,fullpage}
\usepackage{algorithm,algpseudocode}
\usepackage{pgf}

\usepackage{lipsum}
\usepackage{tikz}
\usetikzlibrary{graphs,arrows,automata}
\usetikzlibrary{decorations.markings}
\usetikzlibrary{positioning,chains,fit,shapes,calc}
\usetikzlibrary{shapes.multipart,positioning,fit}
\usetikzlibrary{calc,through}
\usetikzlibrary{decorations}
\usepackage{xifthen}

\usetikzlibrary{external}
\title{$3$-Colouring $P_t$-free graphs without short odd cycles}
\author{\smallskip Alberto Rojas
and Maya Stein\footnote{Supported by FONDECYT/ANID Regular Grant 1183080, by MathAmSud 20MATH-01, by FAPESP-CONICYT Investigaci\'on Conjunta grant C\'odigo 2019/13364-7, and by CONICYT + PIA/Apoyo Basal, C\'odigo AFB170001.}\\   Departamento de Ingenier\'ia Matem\'atica,\\
Universidad de Chile. }
\date{}

\begin{document}

\newtheorem{theorem}{Theorem}[section]
\newtheorem{lemma}[theorem]{Lemma}
\newtheorem{proposition}[theorem]{Proposition}
\newtheorem{observation}[theorem]{Observation}
\newtheorem{corollary}[theorem]{Corollary}
\newtheorem{claim}[theorem]{Claim}
\newtheorem{conjecture}[theorem]{Conjecture}
\newtheorem{defn}[theorem]{Definition}

\newcommand{\Codd}{\mathcal C^{odd}_{\leq t-4}}
\newcommand{\Coddt}{\mathcal C^{odd}_{\leq \ell}}
\maketitle

\begin{abstract}
For any odd $t\ge 9$, we present a polynomial-time algorithm that solves the $3$-colouring problem, and finds a $3$-colouring if one exists, in $P_{t}$-free graphs of odd girth at least $t-2$. In particular, our algorithm works for
 $(P_9, C_3, C_5)$-free graphs, thus making progress towards determining the complexity of $3$-colouring in $P_t$-free graphs, which is open for $t\ge 8$. 

%\noindent \textbf{Keywords:} $3$-colouring, $P_t$-free graph, forbidden subgraph, polynomial time algorithm.
%\noindent \textbf{MSC:} 05C15, 05C85, 05C38.
\end{abstract}

\section{Introduction}

A \emph{$k$-colouring} of a graph $G$ is an assigment of $k$ distinct colours to the vertices of $G$ such that all vertices are coloured and adjacent vertices have distinct colours.
The \emph{$k$-colouring problem} for
a graph $G$ and a natural number $k$ consists in deciding
whether $G$ is $k$-colourable or not. For $k=1$, the problem is equivalent to deciding whether $G$ has an edge, and for $k=2$, it is equivalent to deciding whether $G$ is bipartite, which can be done in linear time using a BFS algorithm. For $k\ge 3$, Karp~\cite{Ka-NPC} showed that the problem is NP-complete.

There are some graph classes for which the $k$-colouring problem is polynomially time solvable, such as perfect graphs~\cite{G-L-S-perf}.
On the other hand, the problem remains NP-complete, %when restricted to planar graphs of maximum degree $4$~\cite{GJ}, and also 
even when  restricted to the class of all triangle-free graphs with bounded maximum degree~\cite{M-P}. % (Maffray and Preissmann~\cite{M-P}).

The class of $H$-free graphs, defined as  the class of all graphs not containing $H$ as an induced subgraph, is a very natural restriction and has appeared in many complexity results.
An extension of the result from~\cite{M-P} is that the $k$-colouring problem is NP-complete for $H$-free graphs, for any graph $H$
containing a cycle~\cite{K-L-col-g, K-K-T-W-col-g}.
Moreover, if $H$ is a forest
with a vertex of degree at least $3$, then $k$-colouring is
NP-complete for $H$-free graphs and $k \geq
3$~\cite{Hol-col,L-G-color}. Combining these results, we obtain that for $k\ge 3$, the $k$-colouring problem is NP-complete in $H$-free graphs, for all graphs $H$ that are {\it not} a disjoint union of 
paths. So naturally, there has been much activity in determining for which $k$ and $t$ the $k$-colouring problem is polynomially time solvable or NP-complete in $P_t$-free graphs, where $P_t$ is the path on $t$ vertices.  (There are also several recent results for graphs not having a disjoint union of paths as an induced  subgraph, see~\cite{mim-width, sP3, updating, P6-rP3, cout,  Paulusma-ISAAC18}.)

Huang~\cite{Huang-col-Pt-free} proved that $4$-colouring
is NP-complete for $P_7$-free graphs, and that $5$-colouring is
NP-complete for $P_6$-free graphs (which implies that $k$-colouring is
NP-complete for $P_6$-free graphs for $k\ge 5$). On the other hand,
Ho{\`a}ng, Kami\'{n}ski, Lozin, Sawada, and
Shu~\cite{H-K-L-S-S-col-P5} showed that $k$-colouring can be
solved in polynomial time on $P_5$-free graphs for any fixed~$k$, and recently, Chudnovsky, Spirkl, and Zhong~\cite{col-P6-free-I, col-P6-free-II} showed that the same is true for $P_6$-free graphs and $k = 4$. This means we have a complete classification of the complexity of
$k$-colouring $P_t$-free graphs for any fixed $k \geq 4$. 

For
$k=3$, less is known.  Randerath
and Schiermeyer~\cite{RS}  and
%Chudnovsky, Maceli, and Zhong~\cite{3-col-tri-P7-2,3-col-tri-P7} and 
Bonomo, Chudnovsky, Maceli, Schaudt, Stein, and Zhong~\cite{B-C-M-S-S-Z-colP7}
gave  polynomial time
algorithms for $3$-colouring $P_6$- and $P_7$-free graphs, respectively.
However, no such algorithm is known for $P_8$-free graphs (although some partial results are known for this class of graphs~\cite{P8}). Interestingly, it is also unknown whether or not there exists $t\in \mathbb N$ such
that $3$-colouring is NP-complete on $P_t$-free graphs. 

Some of the positive results from above are easier to prove if in addition to the path $P_t$, also the triangle, or other cycles, are excluded, and in fact, sometimes this was the first step for proving the full result.
Let  $\mathcal C$ be any fixed family of graphs (here, it will always be a set of cycles). Call a graph $\mathcal C$-free if it does not contain any member of the family $\mathcal C$, and call $G$ $(P_t, \mathcal C)$-free if in addition, $G$ is $P_t$-free.
We refer to~\cite{G-J-P-S-col} and~\cite{Hell-survey-coloring} for an overview on the
complexity of colouring $(P_t, \mathcal C)$-free graphs. 

 In particular, in the case $k=3$, Golovach, Paulusma and Song~\cite{G-P-S-path-cycle} showed  that for any  $t\ge 1$, the $3$-colouring problem is  polynomial time solvable for $(P_t, C_4)$-free graphs. The same authors show that   for every $s\ge 3$ there is a $t=t(s)$ such that every $(P_t, \mathcal C_{\le s})$-free graph is $3$-colourable, where $\mathcal C_{\le s}$ is the family of all cycles up to length $s$. Recall that by the results of~\cite{K-L-col-g, K-K-T-W-col-g}, it is also known that for all  $s\ge 3$, $3$-colouring is NP-complete for $\mathcal C_{\le s}$-free graphs.
 %, and Huang, Johnson and Paulusma~\cite{huang2015narrowing} show that for all $k\ge 4$, $s\ge 5$, $k$-colouring is NP-complete for the class of all $(P_t, \mathcal C_{\le s}\setminus\{C_4\})$-free graphs, where $t=t(s,k)$ depends on $k$ and $s$. 
 
We focus on graphs having no short {\it odd}  cycles. The {\it odd girth} of a graph $G$ is the length of its shortest odd cycle (it is infinite if $G$ is bipartite).
Adhering to the terminology from above, we call a graph $\Coddt$-free, for $\ell\ge 3$ odd, if it has odd girth at least $\ell+2$.

The following is our main theorem.
\begin{theorem}\label{thm:full}
For any odd $t\ge 9$, there is a polynomial time algorithm deciding whether a $(P_{t}, \Codd)$-free graph is $3$-colourable (and giving a 3-colouring if one exists). 
\end{theorem}

In particular, this implies:
\begin{theorem}\label{thm:P9}
There is a polynomial time algorithm deciding whether a $(P_9, C_3, C_5)$-free graph is $3$-colourable (and giving a 3-colouring if one exists). 
\end{theorem}

We remark that the complexity of $3$-colouring $(P_t, C_\ell)$-free graphs is open for $t\ge 8$ and both $\ell =3$, $\ell=5$. 
So, Theorem~\ref{thm:P9} can be seen as a first step to determining the complexity of this problem, or even to resolving the problem for $P_9$-free graphs.

The paper is organised as follows. We give a quick overview of our proof in Section~\ref{sec:over}, and then go through some preliminaries in Section~\ref{prelim}. We present the proof of Theorems~\ref{thm:full} and~\ref{thm:P9} in Section~\ref{sec:algo}. The proof relies on Lemma~\ref{lem:magic}, which is also presented in Section~\ref{sec:algo}. In Section~\ref{sec:stru}, we prove Lemma~\ref{lem:magic} for all pairs $t$, $G$ such that at least one of the following applies: $t=9$; $G$ is
$C_8$-free; or $G$ has an induced $C_{t}$.  In particular, the proof of  Theorem~\ref{thm:P9} is completed after Section~\ref{sec:stru}. 
The more complicated proof for $C_{t}$-free graphs with $t>9$ having an induced $C_8$ is postponed to Section~\ref{sec:C8}.
%\footnote{We chose to present our proof in this way not only to make the easier proof of Theorem~\ref{thm:P9}  independent of the rest, but also because for reading the full proof of Theorem~\ref{thm:full}  it will be easier to first the understand the approach from Section~\ref{sec:stru}, before adding to it the  more complicated structural analysis and the  ideas for the new precolouring  procedure that appear in Section~\ref{sec:C8}.}
  Finally, Section~\ref{sec:time} is devoted to the running-time analysis of our algorithm, the running time being $O(n+m)$.

 \section{Overview of the proof}\label{sec:over}
 Our proof relies on  a structural analysis of a longest odd cycle $C$ of $G$ and it surroundings. This is done  in 
 Section~\ref{sec:stru}, distinguishing the two possible lengths of $C$: $t-2$ or $t$ (these cases are treated in Section~\ref{sec:shortcycle} and~\ref{sec:longcycle}, respectively). In the case that $C$ has length $t-2$, we assume the cycle $C_8$ is forbidden or that $t=9$ (for the time being).
 
 We distinguish different sets $D$, $T$, $T'$, $S$ of neighbours of $C$, according to the location of their neighbours on $C$, and prove all other vertices lie at most at distance 2 from $C$ (unless they are dominated by a non-neighbour, a case which can be safely ignored). We determine which sets can be adjacent to each other, and which cannot. The main results of  the structural analysis in Sections~\ref{sec:shortcycle} and~\ref{sec:longcycle} are summarised in Lemmas~\ref{propdef:c2d+1} and~\ref{propoX}. Most importantly, these lemmas identify a set $X$ of bounded size dominating important parts of the graph. 
 
 In Section~\ref{sec:magic}, we prove Lemma~\ref{lem:magic}, by colouring $V(C)\cup X$ in all possible ways. After doing so, all other vertices will either be adjacent to a coloured vertex, or are {\it reducible}, or belong to {\it reducible bipartite subgraphs}, by which we mean that they can be coloured in a canonical way (see Section~\ref{prelim} for a definition). This allows us to ultimately reduce the problem to a $2$-list-colouring instance, which is known to be polynomial time solvable. This finishes the proof, unless $t=9$ and there is a $C_8$.

We treat that case in Section~\ref{sec:C8}. Our previous approach almost works in the same way. The only problem is that if in addition, $C$ has length $t-2$, the structural analysis leaves us with one difficult case. Namely, there might be a set~$Y^*$ of vertices at distance $2$ from $C$ that cannot be dominated, or made reducible, in a similar way as the other sets. Our solution is that we test all colourings of~$C$, and for each such colouring, we also colour the vertices from $Y^*$, and some of their neighbours and second neighbours `by hand' in a canonical way (that will depend on the colouring of $C$). When doing so, we have to make sure that any vertex  {\it not} coloured `by hand' will not be affected by our colouring - either because it is already coloured in a compatible way, or more generally, because its list of available colours did not decrease. The main challenge in Section~\ref{sec:C8} is how to choose the
vertices that will be coloured together with $Y^*$, in way that other vertices will not be affected. This is achieved by a structural analysis.
We can then proceed as before. That is, we test all colourings of~$X$, colour the reducible vertices and subgraphs, and thus reduce the problem to a $2$-list-colouring instance.

\section{Preliminaries}\label{prelim}
\subsection{Basic definitions}

For $n\in\mathbb N$, let $[n]=\{0,1,...,n\}$. A graph $G$ has
vertex set $V(G)$ and edge set $E(G)$.  We only need to consider  connected graphs $G$, since one can solve the $3$-colouring problem on the components. We call $G$ (or one of its components $K$) \emph{trivial} if it has only one vertex.

If $G$ does not contain another graph $H$ as an induced subgraph, we say $G$ is {\it $H$-free}. If $\mathcal H$ is a family of subgraphs, and $G$ does not contain any of the graphs in $\mathcal H$ as an induced subgraph, we say $G$ is $\mathcal H$-free. Sometimes we write $(H, \mathcal H)$-free %(or $(H_1, H_2, \mathcal H)$-free)
 to mean that the graph is $\{H\}\cup \mathcal H$-free. %(or $\{H_1, H_2\}\cup \mathcal H$-free).

Let $v \in V(G)$ and $A, B \subseteq V(G)$. Then $N_B(v)$ is the set
of all neighbours of $v$ in $B$, and  $N_B(A)$ is the set of all neighbours
of vertices of $A$ in $B\setminus A$. If $B=V(G)$, we omit the subscript. 
Note that $N(\emptyset)=\emptyset$.
Also, $G[A]$ denotes the subgraph of $G$ induced by $A$, and $G-v$ is the graph obtained from $G$ by deleting $v$ and its adjacent edges.  We call a vertex $v$ {\it complete} to  $A$ if all $v$-$A$ edges are present, and if this holds for all $v\in B$, we say $B$ is complete to $A$.

A \emph{stable set} is a subset of pairwise non-adjacent vertices.
A graph $G$ is \emph{bipartite} if $V(G)$ can be partitioned into
two stable sets. 
A vertex $w\in V(G)$ {\it dominates} another vertex $v\in V(G)$ if $N(v)\subseteq N(w)$.
The following lemma is well-known and easy to prove.

\begin{lemma} \label{EDGESIMPLE}
Let $G, H$ be graphs such that $G$ is $H$-free, and let $v\in V(G)$. Then  $G-v$ is $H$\text{-free}, and 
if $v$ is  dominated by one of its non-neighbours, then \[\text{$G$ is $3$\text{-colourable} if and only if $G-v$ is $3$\text{-colourable}.}\]
\end{lemma}

\subsection{Palettes, updates and reducibility}\label{sec:redu}

In the context of 3-colouring, we call a family $\mathcal L$  of lists $L :=\{ L(v): \ v\in V(G) \}$, where $L(v)\subseteq\{1,2,3\}$ for every $v\in V(G)$, a {\it palette} of the graph $G$.  We call $L$ {\it feasible} if no vertex has an empty list. A subpalette $L'$ of $L$ is a palette such that $L'(v)\subseteq L(v)$ for each $v\in V(G)$. A graph $G$ and a palette $L$ of $G$ will often be written as a pair $(G,L)$. 
We say that $(G,L)$ is colourable if there is a subpalette $L'$ of $L$ such that $|L'(v)|=1$ for every vertex $v\in V(G)$, and $L'(u)\neq L'(w)$ for every edge $uw\in V(G)$.
We sometimes call a vertex with a list of size 1 a {\it coloured} vertex.

 \textit{Updating} the palette $L$ means obtaining a subpalette $L'$ of~$L$ by subsequently, for any vertex $v$ having a list  of size $1$ in the current palette, deleting its colour from the list of each neighbour $w$ of $v$. (For instance, if in a path all but one vertex $v$ have the list $\{1,2\}$, and $v$ has list $\{1\}$, then updating this palette gives a colouring of the path.)
If updating a palette $L$ of a graph $G$ results in the same palette $L$, we will say that $L$ is \textit{updated}.

Given $(G,L)$, call a vertex $v\in V(G)$  \textit{reducible (for colour $\alpha$)}  if $\alpha\in L(v)$ and $\alpha\notin L(w)$ for each $w\in N(v)$.
A trivial subgraph $K=\{v\}$  of $G$ is called reducible if $v$ is reducible. The subpalette $L'$ obtained from $L$ by setting $L'(v):=\alpha$ and keeping all other lists will be called the {\it reduction palette} for the reducible component $K$.

 A bipartite subgraph $K$of $G$, with partition classes $U_1$, $U_2$, is called {\it reducible (for colours $\alpha_1, \alpha_2$)}  if there are distinct colours $\alpha_1,\alpha_2$ such that for $j=1,2$,  $$\alpha_j\in \bigcap_{u\in U_j}L(u)\setminus \bigcup_{w\in N(U_{j})\setminus U_{3-j}}L(w).$$  The subpalette $L'$ obtained from $L$ by setting 
$L'(u) = \{\alpha_j\}$  for $j=1,2$ and  all $u\in U_j$, and keeping all other lists will be called the {\it reduction palette} for the reducible component $K$.

We leave the proof of the following lemma as an exercise to the reader.

\begin{lemma}
Let $G$ be a graph and let $L$ be palette of $G$. Suppose $K$ is a reducible subgraph of $G$, and $L'$ is the corresponding reduction palette.
Then $(G,L)$ is 3-colourable if and only if $(G,L')$ is 3-colourable. 
\end{lemma}

\section{The proof of Theorem~\ref{thm:full}}\label{sec:algo}

The heart of our proof is the following lemma, which we will prove in Sections~\ref{sec:stru} and~\ref{sec:C8}.
In order to state the lemma easily, let us define, for a graph $G$ and a palette $L$, the set $V_3(G,L):=\{ v\in V(G) : |L(v)|= 3 \}$.

\begin{lemma}\label{newlem:magic}\label{lem:magic}
Let $G$ be a $(P_{t},\Codd )$-free connected graph. Then there is a set $\mathcal L$ of updated feasible palettes such that
\begin{itemize}
\item  $|\mathcal L|\le 2^{O(d)}$ and  $\mathcal L$ can be found in polynomial time;
\item $G$ is $3$-colourable if and only if $(G,L)$ is $3$-colourable for some $L\in\mathcal L$; and
\item for each $L\in\mathcal L$, every component of $G[V_3(G,L)]$ is  reducible.
\end{itemize}
\end{lemma}

%\begin{lemma}\label{lem:magic}
%Let $G$ be a $(P_{2d+3},C_{\leq 2d-1})$-free graph having no induced cycle of length $8$. Then there is a set $A \subseteq V(G)$, with $|A|\le 8d+14$, having the following property.  If  $L$ is an updated feasible palette of $G$ such that $|L(a)|=1$ for every $a\in A$, then each component of $G[V_3(G,L)]$ is reducible.\\
%Furthermore, the set $A$ can be found in polynomial time.
%\end{lemma}

We will also need a result on
the list-colouring problem with lists of size at most 2. Given a graph $G$
and a finite list $L(v) \subseteq \mathbb{N}$  for each vertex
$v\in V(G)$ (with no restriction on the total number of colours)
 the {\it list-colouring problem} asks for a
a colouring of all vertices with colours from their lists. If $|L(v)| \leq 2$ for each vertex $v\in V(G)$, the
problem can be solved in $O(|V(G)|+|E(G)|)$ time, by reducing the $2$-list-colouring instance to a $2$-SAT instance \cite{E-R-T-2-lc,Viz-color}.

%With this result and Lemma~\ref{lem:magic} at hand, we are ready to present the algorithm that proves our main result.
%
%\begin{algorithm}
%\caption{Coloring a $(P_{2d+3},\Codd)$-free graph $G$}\label{algC7}
%\begin{algorithmic}[1]
%\State Given $G$, take $\mathcal{L}$, family of palettes from Lemma~\ref{lem:magic}.
%\For{each reducible component $K$ of $G[V_3(G,L)]$}
%\State{Let $L'$ be the reduction palette for $K$.
%\State $L\gets L'$.}
%\EndFor
%\For{each $L\in \mathcal{L}$}
%\State Solve $2$-list-colouring instance $(G,L)$.
%\If{$(G,L)$ is $2$-list-colourable}
%\State \textbf{return} $True$ and the colouring.
%\EndIf
%\EndFor
%\State \textbf{return} $False$
%\end{algorithmic}
%\end{algorithm}
With this result and Lemma~\ref{lem:magic} at hand, we are ready to present the proof of our main result.
\begin{proof}[Proof of Theorem~\ref{thm:full}]
Given $G$, we apply Lemma~\ref{lem:magic} to obtain a set  $\mathcal{L}$ of palettes. For each $L\in\mathcal L$, and each reducible component of $(G,L)$, we consider the reduction palette $L'$ for $L$ and solve the $2$-list-colouring instance $(G,L')$. This either gives us a valid $3$-colouring for one of the palettes $L$, and thus for our instance $G$, or proves that no such colouring exists.
\end{proof}

The running time of the algorithm and the time of finding the set $\mathcal L$ from Lemma~\ref{lem:magic} will be analysed in Section~\ref{sec:time}.

Sections~\ref{sec:stru} and~\ref{sec:C8} are devoted to the proof of Lemma~\ref{lem:magic}. As discussed above we first concentrate on the easier case that either $t=9$, or $G$ is $C_8$-free, or $G$ has an induced $C_{t}$, which we deal with in Section~\ref{sec:stru}, and leave the more complicated general case for Section~\ref{sec:C8}. 
For simplicity of notation, let $\mathcal G^*$ denote the class of all $(P_{t}, \Codd)$-free graphs such that
at least one of the following holds:
\begin{itemize}
\item $t=9$; or
\item $G$ is $C_8$-free; or
\item $G$ has an induced $C_{t}$.
\end{itemize}

\section{The proof of Lemma~\ref{lem:magic} for  $G\in\mathcal G^*$}\label{sec:stru}

In this section, we will prove Lemma~\ref{lem:magic} for all $G\in\mathcal G^*$. In particular, as stated earlier, this section is all that is needed for the proof of Theorem~\ref{thm:P9}. 

The  section is organised as follows. First, we investigate the structure of the graph~$G$, proving some basic observations in Section~\ref{basecyc}, and then distinguishing two cases according to whether or not~$G$ has  an induced cycle of length $t$. The structural analysis for the case that~$G$ is $C_{t}$-free is treated in Section~\ref{sec:shortcycle}, while the case that~$G$ does have a cycle of length $t$  is treated in Section~\ref{sec:longcycle}. 
All structural properties we find in these two sections will be conveniently summarised in two lemmas, Lemma~\ref{propdef:c2d+1} (at the end of Section~\ref{sec:shortcycle}) and Lemma~\ref{propoX} (end of Section~\ref{sec:longcycle}). Then, in Section~\ref{sec:magic}, where we will prove Lemma~\ref{lem:magic} for  $G\in\mathcal G^*$, we only need to refer to these two lemmas.

\subsection{The base cycle and its neighbours}\label{basecyc}

%Let $G\in \mathcal G^*$.
We start with a general definition describing the neighbours and second neighbours of a cycle~$C$ of~$G$.

\begin{defn}\label{def:DTTS}
Let $\ell\in \mathbb{N}$  and let $C=c_0, c_1,...,c_{\ell}, c_0$ be a cycle in $G$. For $i\in[\ell]$, define 
\begin{itemize}
    \item $D_i:=\big\{v\in N(C): vc_j \in E(G)\text{ if and only if } j= i\big\}$;
    
    \item  $T_i:=\big\{v\in N(C): vc_j \in E(G)\text{ if and only if } j\in\{ i, i+2\}\big\}$;
    
    \item  $T'_i:=\big\{v\in N(C): vc_j \in E(G)\text{ if and only if } j\in\{ i, i+4\}\big\}$; and
    
    \item  $S_i:=\big\{v\in N(C): vc_j \in E(G)\text{ if and only if } j\in\{ i, i+2, i+4\}\big\}$.
\end{itemize}
Set $$D:=\bigcup_{i\in[\ell]}D_i, \ \  T:=\bigcup_{i\in[\ell]}T_i, \ \ T':=\bigcup_{i\in[\ell]}T'_i\text{ \ and \ }S:=\bigcup_{i\in[\ell]}S_i,$$ and let  $Y$ denote the set of vertices located at distance two from $C$.
\end{defn}

Clearly, since we can assume $G$ is not bipartite (as otherwise the $3$-colouring problem is well known to be polynomial), $G$ has an odd cycle $C$. Because of the forbidden subgraphs, the length of the longest odd cycle is either $t-2$ or $t$. We will see now that in each of these two cases, 
 $V(G)$ can be  decomposed into some of the sets that Definition~\ref{def:DTTS} gives for $C$.

\begin{claim}\label{lem:structure}
Suppose that no vertex in $G$ is dominated by any of its non-neighbours.\footnote{Because of Lemma~\ref{EDGESIMPLE}, we need not  worry about dominated vertices, and Claim~\ref{lem:structure} is the only place where we need to exclude them. In an algorithmic implementation of our method, one will ignore such vertices if detected by Claim~\ref{lem:structure}, and colour them at the very end (in case a colouring is found).
} Let $C$ be a longest odd cycle in $G$.
\begin{enumerate}[(a)]
 \item If $|V(C)|=t-2$,  then 
$V(G)=V(C) \cup N(C) \cup Y$,
with $N(C)=D\cup T$.
\item If $|V(C)|=t$,  then 
$V(G)=V(C) \cup N(C) \cup Y$,
with $N(C)=T \cup T' \cup S$. 
\end{enumerate}
Furthermore, for each $i\in [|V(C)|-1]$, each of the sets  $D_i$, $T_i$, $T'_i $ and $S_i$ is  stable.
\end{claim}

\begin{proof}
Since $G$ has no odd induced cycles of length up to $t-4$, and no induced $P_{t}$, it is not hard to see that in case (a), $N(C)=D\cup T$, and in case (b), $N(C)=T \cup T' \cup S$. Moreover, each of the sets $D_i$, $T_i$, $T'_i $ and $ S_i$, for every $i\in [t-3]$, or $i\in[t-1]$, is stable, as~$G$ is triangle-free. 

It only remains to show that no vertex of $G$ lies at distance $3$ from $V(C)$. Assume otherwise, and let $v_1v_2v_3$ be an induced path such that $v_i$ lies at distance $i$ from $V(C)$. Suppose $C=c_0, c_1,...,c_{\ell}, c_0$, and that $v_1$ is adjacent to $c_0$, but  not adjacent to any  $c_i$ with $i\le t-5$.
Since~$v_1$ does not dominate $v_3$, there is a neighbour $w$ of $v_3$ that is not a neighbour of~$v_1$, and since $G$ is triangle-free, $w$ is not adjacent to $v_2$. Furthermore, since $v_3$ is at distance $3$ from~$C$, we know that~$w$ does not have any neighbours on $C$. Thus $w, v_3, v_2, v_1, c_0, c_{1}, c_{2}, ...,c_{t-5} $ is an induced $P_{t}$, which is impossible.
\end{proof}

We can deduce some useful information on the components of $G[Y]$.

\begin{claim}\label{Ybipsamenbhds}
In the situation of Claim~\ref{lem:structure}, every nontrivial component $K$ of $G[Y]$ is bipartite and vertices of the same bipartition class of $K$ have identical neighbourhoods in $N(C)$.
\end{claim}

\begin{proof}
Since $G$ is $(P_{t}, C_3)$-free, we know that for every induced path $y_1y_2y_3$ of length 3 in $G[Y]$, vertices $y_1$ and $y_3$ have the same neighbours in $N(C)$. In order to see that $G[Y]$ is bipartite, apply the observation from the previous sentence to any odd cycle in $G[Y]$, thus generating a $C_3$ and therefore, a contradiction.
\end{proof}

\subsection{The structure of $G$ if it is $C_{t}$-free }\label{sec:shortcycle}

In this subsection we  asume that  $G$ is  $C_{t}$-free, but does have an induced cycle  $C=c_0, c_1, ..., c_{t-3}, c_0$ of length $t-2$.  Let $D$, $T$ and $Y$ be the sets given by Definition~\ref{def:DTTS} and Claim~\ref{lem:structure} for cycle $C$. Indices are  taken modulo $t-2$.
 
 For understanding the structure of $G[N(C)\cup Y]$, we first analyse the edges between $Y$ and~$N(C)$.

\begin{claim}\label{DY...}
No vertex from $Y$ can have neighbours both in $Y$ and  in $D$.
\end{claim}

\begin{proof}
This follows directly from the fact that $G$ is $(P_{t}, C_3)$-free.
\end{proof}

\begin{claim}
\label{lem:y_i^*}
For each $y\in Y$ there is an $i\in[t-3]$ such that at least one of the following holds:
\begin{enumerate}[(a)] \item $N(y)\setminus Y\subseteq D_{j}\cup  T_{i}\cup T_{i+2}$ for some $j\in \{i-2, i, i+2, i+4\}$; or
\item  $t=9$ and 
$N(y)\subseteq D_{i}\cup T_{i}\cup T_{i+2}\cup D_{i+4}$.
\end{enumerate}
\end{claim}

\begin{proof}
Because of the forbidden cycles (and observing that if $t=9$ then $\ell-4=\ell+3$), we know that if $y\in N(D_\ell)$, then 
 all neighbours of $y$ in $D$  lie in $D_\ell\cup D_{\ell +4}$ or in $D_{\ell-4}\cup D_{\ell}$, and
 all neighbours of $y$ in $T$  lie in $T_{\ell-4}\cup T_{\ell-2}\cup  T_\ell\cup T_{\ell +2}$.
Moreover, if $y$ has a neighbour in $T_h$, then $N(y)\cap T$ is contained either  in
$T_h\cup T_{h+2}$ or  in $T_{h-2}\cup T_{h}$. This gives (a), unless
 $y$ has neighbours in both $D_i$ and in $D_{i+4}$. However, in that case,  the neighbours of $y$ in $T$ must belong to $T_{i}\cup T_{i+2}$, and hence, there is an induced $C_8$ going through $y$, which is only allowed if $t=9$. Putting these observations together, and using Claim~\ref{DY...} for (b), the statement follows.
\end{proof}

Let $Y'_i$ be the set of all $y\in Y$ which have neighbours in both $D_i$ and $D_{i+4}$. (By Claim~\ref{lem:y_i^*}, this set is empty if $t>9$.) Observe that if there are independent edges $xy$, $x'y'$ such that $y, y'\in Y'_i$, $x\in D_i$, $x'\in D_{i+4}$, then $y, x, c_{i}, c_{i+1}, ..., c_{i+4}, x', y'$ is an induced path of length $t$ (i.e.~of length $9$), which is impossible. Therefore, for any two vertices $y,y'\in Y'_i$, we have that either $N(y)\cap D_i\subseteq N(y')\cap D_i$ or $N(y')\cap D_{i+4}\subseteq N(y)\cap D_{i+4}$ (or both). Choosing two vertices $y_i, y'_i$ in each set $Y'_i$ such that $N(y)\cap D_i$ and $N(y')\cap D_{i+4}$ are inclusion-minimal,  then choosing a neighbour $q_i$ of $y_i$ in $D_i$, and a neighbour $q'_i$ of $y'_i$ in $D_{i+4}$, and letting $Q:=\bigcup_{i\in [t-3]}\{q_i, q_i'\}$,
we obtain the following claim.
\begin{claim}\label{d=3}
There is a set $Q$ of size at most $14$ such that every vertex in $\bigcup_{i\in[t-3]}Y'_i$ has a neighbour in $Q$.
\end{claim}

Let us now analyse possible edges inside $Y$.

\begin{claim}\label{lem:yy'}
For any $i\in[t-3]$, if $yy'\in E(G[Y])$ and $y\in N(T_i)$, then
\[N(y')\setminus Y\subseteq \bigcup_{j\in\{i-3, i-1, i+1, i+3\}}( T_{j}\cup T_{j+2}).\]
\end{claim}

\begin{proof}
This follows  from the fact that $G$ has no induced odd cycles of length up to $t-4$.
\end{proof}

Claims~\ref{DY...} and~\ref{lem:yy'} immediately imply the following.
\begin{claim}\label{cor:y}
If $y, y'\in Y$ are adjacent, then 
\begin{enumerate}[(i)]
\item there is a vertex $c$ on $C$ such that $N(y)\subseteq N(c)$; and
\item  there are two consecutive vertices $c,c'$ on $C$ such that $z\in N(N(y))$ and $z'\in N(N(y'))$.
\end{enumerate}
\end{claim}

We now define a set $W$ of certain vertices of $Y$.

\begin{defn}\label{defiWC2d+1}
Let $W\subseteq Y\cap N(T)$ be the set of all vertices $y\in Y$ for which there is an $i\in [t-3]$ such that  $y\in N(T_i)$ and one of the following holds:
\begin{enumerate}[(i)]
    \item  $y \in N(D_{i-2}\cup D_{i+4})$; or
    \item $y$ has a neighbour in $Y\cap  N(T_{i-3}\cup T_{i+3})$.
\end{enumerate}
\end{defn}

Observe that by Claim~\ref{DY...}, the vertex set of any component of $G[Y]$ is either contained in $W$ or disjoint from $W$.

\begin{claim}\label{lem:completeYW}
Let $y, y'\in Y\setminus W$ with $yy'\in E(G)$. Then for any $z\in N(y)\setminus Y$ and $z'\in N(y')\setminus Y$, we have that $zz'\in E(G)$.
\end{claim}

\begin{proof}
Suppose otherwise. By Claim~\ref{lem:yy'} there is an index $i\in [t-3]$ such that $z\in T_i$ and $z'\in T_{i+1}$ (after possibly changing the roles of $z$ and $z'$). Then
$z, y, y', z', c_{i+3}, c_{i+4},  ... , c_{i-1}, c_i, z$
is an induced cycle of length $t$, a contradiction.
\end{proof}

It turns out that  $W$  is dominated by a set $X$ of bounded size. 

\begin{claim}\label{lem:colourW2d+1}
There exist a set $R \subseteq N(C)$ such that $W \subseteq N(R)$ and $|R|\le 2t-4$.
\end{claim}

\begin{proof}
Let  $i\in [t-3]$. We claim that for any two vertices $y,z\in W\cap N(T_i)$ we have
\begin{equation} \label{prop1i}
    N_{T_i}(y) \subseteq N_{T_i}(z) \text{, or } N_{T_i}(z) \subseteq N_{T_i}(y)\text{, or }N(y)\cap D=N(z)\cap D\neq \emptyset.
\end{equation}
Then we can take,
for each $i\in [t-3]$ with $T_i\neq\emptyset$, a vertex $y\in W$, with the property that among all such choices, $N_{T_i}(y)$ is inclusion-minimal, and choose an arbitrary vertex $a_i\in N_{T_i}(y)$. If $N(y)\cap D\neq \emptyset$, we also choose a arbitrary vertex $b_i\in N(y)\cap D$. Then by~\eqref{prop1i},  the set $R$ consisting of all $a_i$ and all existing $b_i$ is as desired.

It remains to prove~\eqref{prop1i}. For contradiction suppose \eqref{prop1i} fails for vertices $y,z\in W\cap N(T_i)$. Then  there are  $t_1\in N_{T_i}(y) \setminus N_{T_i}(z)$ and $t_2\in N_{T_i}(z)  \setminus N_{T_i}(y)$.

We distinguish two cases. First assume  $z$ is as in Definition~\ref{defiWC2d+1} \textit{(i)}, that is, $z\in N(D_{i-2}\cup D_{i+4})$. Then there is $d\in N_{D_{i-2}\cup D_{i+4}}(z)$, say $d\in D_{i+4}$ (the other case is symmetric). Since we assume~\eqref{prop1i} does not hold, we may assume that $d\notin N(y)$ (after possibly swapping the roles of $y$ and $z$). Consider the induced path $y, t_1, c_{i+2}, t_2, z, d, c_{i+4}, c_{i+5}, c_{i+6}, ... , c_{i-1}$, which has length $t$,  a contradiction.

Now assume $z$ is as in 
Definition~\ref{defiWC2d+1} \textit{(ii)}. Then $z$ has a a neighbour $z'\in Y\cap N(T_{i-3}\cup T_{i+3})$, say $z'\in Y\cap N(T_{i+3})$. Let $t_3\in N_{T_{i+3}}(z')$.
 Consider the path $y, t_1, c_{i+2}, t_2, z, z', t_3, c_{i+5}, c_{i+6}, c_{i+7}, ... , c_{i-1}.$
As this path is induced ($yt_{3}\notin E(G)$ because of Claim~\ref{lem:y_i^*}) and has length $t$, we again obtain a contradiction to $G$ being $P_{t}$-free. 
This proves~\eqref{prop1i}, and thus concludes the proof of the lemma.
\end{proof}

We now resume the main results of this subsection in the following lemma.

\begin{lemma}\label{propdef:c2d+1}
There is a set $X\subseteq N(C)$ with $|X|\le 2t+10$ such that for every component $K$ of $G[Y]$ one of the following holds.
\begin{enumerate}[(I)]
    \item $V(K)\subseteq N(X)$; 
    \item $K$ is trivial and there is a vertex $c$ on $C$ such that $N(K)\subseteq N(c)$; or
    \item $K$ is bipartite with partition classes $U_1, U_2$, and 
    \begin{enumerate}[(A)]
    \item there are vertices $b_1, b_2\in V(C)$ such that $N(U_j)\subseteq N(b_j)$  for $j=1,2$; \label{cseesall} \item there are $x_j\in N(N(U_j))\cap V(C)$   for $j=1,2$ such that $x_1x_2\in E(G)$;\label{adjacencies}
   and
    \item $N(U_1)\setminus Y$  is complete to $N(U_{2})\setminus Y$.\label{completeneigh}
\end{enumerate}
\end{enumerate}
\end{lemma}
\begin{proof}
Apply Claim~\ref{lem:colourW2d+1} to obtain a set $R$ of size at most $2t-4$ that dominates $W$. Apply Claim~\ref{d=3} to obtain a set $Q$ of size at most $14$ that dominates $\bigcup_{i\in [t-3]}Y'_i$.  Set $X:=R\cup Q$. 

 Now, consider any vertex $y\in Y\setminus  W$.
By Claim~\ref{lem:y_i^*}, there is  an $i\in[t-3]$ such that all neighbours of $y$ in $N(C)$ belong to $D_i\cup T_{i-2}\cup T_{i}$. If $y$ is isolated in $G[Y]$, we take $c(y):=c_i$ which is as required for (II). 
If $y$ is not isolated in $G[Y]$, then Claim~\ref{Ybipsamenbhds} implies that $K$ is bipartite and the vertices of each partition class have identical neighbourhoods in $N(C)$.  So by Claim~\ref{cor:y}  and  Claim~\ref{lem:completeYW}, we have (III).
\end{proof}

\subsection{The structure of $G$ if it contains a $C_{t}$}\label{sec:longcycle}

In this subsection, we assume that
 $G$ has  a cycle $C=c_0,c_1,...,c_{t-1}$ of length $t$, and we let $T $, $ T'$ and $S$ be the sets given by Definition~\ref{def:DTTS} and Claim~\ref{lem:structure} for cycle $C$, while  $Y$ denotes the set of vertices located at distance two from $C$. 
In this subsection, all indices are taken modulo $t$.

We start by understanding some basic  adjacencies.

\begin{claim}\label{lem:edgesinN(C)_A}
If  $uv\in E(G[N(C)])$ and $i\in[t-1]$, then the following holds.
\begin{enumerate}[(a)]
\item If $u\in T_i$ then $v\in T_{i+3}\cup \bigcup_{j\in\{i+1, i-1, i-3\}}\big (T_{j}\cup  T'_{j}\cup
S_{j}\big )$; and
\item if $u\in T'_i$ then $v\in T_{i+3}\cup \bigcup_{j\in\{i+1,  i-1\}}\big (T_{j}\cup  T'_{j}\cup
S_{j}\big )$.
\end{enumerate}
\end{claim}

\begin{proof}
Any case not covered would lead to an induced odd cycle of length $\le t-4$ on vertices from $V(C)\cup \{ u,v\}$. But such a cycle is forbidden in $G$. 
\end{proof}

Claim~\ref{lem:edgesinN(C)_A} easily implies the following claim.

\begin{claim}\label{corollo}
If  $uv\in E(G[N(C)])$, then there is an edge between $N(v)\cap V(C)$ and $N(u)\cap V(C)$. 
\end{claim}

We now turn to the edges between $N(C)$ and $Y$.
Our first claim follows directly from the fact that $G$ is $P_{t}$-free.

\begin{claim}\label{noYTedges}
There are no edges between $Y$ and $T$.
\end{claim}

There may, however, exist edges between $Y$ and $T'\cup S$. Let us see what we can say about these.

%\begin{claim}\label{lem:edgesfromYtoN(C)}
%If $y\in Y\cap N(T'_i\cup S_i)$, for some $i\in[2d+2]$, then $N(y)\subseteq Y\cup \bigcup_{j\in\{i, i+2, i-2\}} (T'_j\cup S_j).$
%\end{claim}

%Claims~\ref{noYTedges} and~\ref{lem:edgesfromYtoN(C)} have the following straightforward corollaries.

\begin{claim}\label{co}
For each $y\in Y$ there is an  $i\in[t-1]$ such that 
\begin{enumerate}[(a)]
\item $N(y)\subseteq Y\cup \bigcup_{j\in\{i, i+2\}} (T'_j\cup S_j)$; and\label{coroindexi}
\item $c_{i+2}$ is complete to $N(y)\setminus (Y\cup T'_{i})$ and $c_{i+4}$ is complete to $N(y)\setminus (Y\cup T'_{i+2})$. \label{coroTheNeighbour}
\end{enumerate}
\end{claim}

\begin{proof}
Note that if $y\in Y\cap N(T'_i\cup S_i)$, for some $i\in[t-1]$, then, by Claim~\ref{noYTedges} and since there are no induced odd cycle of length at most $t-4$ on vertices from $V(C)\cup N(y) \cup \{ y\}$, we obtain that $N(y)\subseteq Y\cup \bigcup_{j\in\{i, i+2, i-2\}} (T'_j\cup S_j).$ This gives (a), which implies (b). 
\end{proof}

\begin{claim} \label{COMP}
Let  $y\in Y$ and let $i$ be as in Claim~\ref{co}. If $T'_i\neq\emptyset\neq T'_{i+2}$ and $y\in N(T'_i) \cup N(T'_{i+2})$, then $y$ is complete to $T'_i \cup T'_{i+2}$.
\end{claim}

\begin{proof} 
Assume $y\in N(T'_i)$ (the other case is symmetric). Let $t_1 \in N(y)\cap T'_i$ and let $t_2 \in T'_{i+2}$. Consider the path $y, t_1, c_{i+4}, c_{i+3}, c_{i+2}, t_2, c_{i+6}, c_{i+7},...,c_{i-2},c_{i-1}.$ This path has length $t$,  so it cannot be induced. Note that the only one possible chord is $yt_2$. As the vertex $t_2$ was chosen  arbitrarily  from $T'_{i+2}$, this means $y$ is complete to $T'_{i+2}$.

Changing the roles of the sets $T'_{i}$ and $T'_{i+2}$ in the above argument we see that $y$ is also complete to $T'_{i}$.
\end{proof}

Claim~\ref{COMP} enables us to prove the following claim.

\begin{claim}\label{coroM}
There is a set $M\subseteq N(C)$ with $|M|\le 2t$ such that 
\begin{enumerate}[(i)]
\item if  $i\in[t-1]$ is such that $T'_{i-2}\cup T'_{i+2}\neq\emptyset$, then $N(T'_i)\cap Y\subseteq N(M)$; and
\item $M\cap T'_i\neq \emptyset$ for all $i$ with $T'_i\neq\emptyset$.\label{allT'i}
\end{enumerate}
\end{claim}
\begin{proof}
Choose for $M$ one vertex  from each of the non-empty sets $T'_i$, for $i\in[t-1]$. Then, $|M|\le 2t$, and (ii) clearly holds. In order to see (i), note that by Claim~\ref{COMP} any neighbour of $T'_i$ in $Y$ is complete to $T'_i$, and thus sends an edge to $x_i$.
\end{proof}

We now check  edges in $Y$ and their neighbours in $N(C)$.

\begin{claim}\label{lem:edgesinY_A}
For any $i\in[t-1]$, if $yy'\in E(G[Y])$ and $y\in N(T'_i\cup S_i)$, then
\[N(y')\setminus Y\subseteq \bigcup_{j\in\{i-3, i-1, i+1, i+3\}}(T'_{j}\cup
S_{j}).\]
\end{claim}

\begin{proof}
Any case not covered would lead to an induced odd cycle on at most $t-4$ vertices from $V(C)\cup N(y) \cup N(y')\cup \{ y,y'\}$. 
\end{proof}

Now, we will identify a useful subset $W \subseteq Y$.

\begin{defn}\label{defiW}
We define $W\subseteq Y\cap N(T'\cup S)$ as the set of endvertices of all edges $yz$ such that there is $i\in [t-1]$ with $y\in Y\cap N(T'_i\cup S_i)$  and $z\in Y\cap \bigcup_{j\in\{i-3, i+3\}} N(T'_{j}\cup
S_{j})$.
\end{defn}

Note that by Claim~\ref{Ybipsamenbhds}, any component of $G[Y]$ has either all or none of its vertices in~$W$. 

%Claim~\ref{lem:edgesinY_A} has the following corollary.

\begin{claim}\label{corooo}
For any $i\in[t-1]$, $yy'\in E(G[Y])$ and $y\in N(T'_i \cup S_i)\setminus W$ the following hold.
\begin{enumerate}[(a)]
\item If $N(y)\cap N(T'_{i+2})\neq\emptyset$, then $N(y')\setminus Y\subseteq  T'_{i+1}\cup S_{i+1}$.\label{itemthree}
\item Each $w \in N(N(y')\setminus Y)\setminus Y$ is adjacent to some neighbour of $N(y)$ in $V(C)$.\label{itemone}
\item There are  consecutive vertices $c, c', c''$ on $C$ such that $N(y)\setminus Y\subseteq  N(c)\cup N(c'')$ and $N(y')\setminus Y\subseteq N(c')$.
\item If $y, y'\notin \bigcup_{i\in [t-1]}(N(T'_i)\cap N(T'_{i+2}))$, then there  are consecutive vertices $c, c'$ on $C$ such that $N(y)\setminus Y\subseteq N(c)$ and $N(y')\setminus Y\subseteq N(c')$.\label{itemtwo}
\end{enumerate}
\end{claim}
\begin{proof}
Item~\eqref{itemthree}  is 
 straightforward from Claim~\ref{lem:edgesinY_A}. For $w\in N(N(y')\setminus Y)\cap V(C)$, item~\eqref{itemone} follows from Claim~\ref{lem:edgesinY_A} and the definition of $W$, and for 
 $w \in N(N(y')\setminus Y)\cap N(C)$,  item~\eqref{itemone} follows from Claim~\ref{corollo}. 
 For~(c) and~\eqref{itemtwo}, we use Claim~\ref{co}.
\end{proof}

We now show that $W$  is dominated by a set of bounded size,~$B$. 

\begin{claim}\label{lem:colourW}
There is a set $B\subseteq N(C)$ such that $W \subseteq N(B)$ and $|B|\le  t$.
\end{claim}
\begin{proof}
Let $yz, y'z'\in E(G[W])$ belong to different components from $Y$, and assume $i\in [t-1]$ is such that $y,y'\in N(T'_i\cup S_i)$ and $z,z' \in N(T'_{i+3}\cup S_{i+3})$.
We claim that at least one of the following holds.
\begin{equation} \label{prop1}
    N_{T'_i\cup S_{i}}(y) \subseteq N_{T'_i\cup S_{i}}(y') \text{, or } N_{T'_i\cup S_{i}}(y') \subseteq N_{T'_i\cup S_{i}}(y).
\end{equation}
Indeed, if \eqref{prop1} is false, then  there are vertices $s_1\in N_{T'_i\cup S_{i}}(y) \setminus N_{T'_i\cup S_{i}}(y')$, $s_2\in N_{T'_i\cup S_{i}}(y')  \setminus N_{T'_i\cup S_{i}}(y)$, and $s_3\in N_{T'_{i+3}\cup S_{i+3}}(z')$. Consider the path $y, s_1, c_{i+4}, s_2, y', z', s_3, c_{i+7}, c_{i+8}, c_{i+9}, ... , c_{i-1}.$
This path has length $t$, and is induced, a contradiction. This proves~\eqref{prop1}.

Now, for each $i\in [t-1]$ with $T'_i\cup S_i\neq\emptyset$, take an edge $yz\in E(G[W])$, with  $y\in Y\cap N(T'_i\cup S_i)$  and $z\in Y\cap N(T'_{i+3}\cup
S_{i+3})$, and with the property that among all such choices, $N_{T'_i\cup S_{i}}(y)$ is in\-clusion-mini\-mal. Choose an arbitrary vertex $x_i\in N_{T'_i\cup S_{i}}(y)$. Then by~\eqref{prop1},  $B:=\{x_0, x_1, ..., x_{t-1}\}$ is as desired.
\end{proof}

We now show an interesting fact about the neighbourhoods of components of $G[Y]$.

\begin{claim}\label{nononono}
Let   $K$ be a bipartite component of $G[Y\setminus W]$, with partition classes $U_1, U_2$, and let $w\in N(U_1)\cap N(C)$. If there is no edge from $w$ to $N(U_2)\cap N(C)$, then 
 $N(w)\cap Y\subseteq V(K)$.
\end{claim}

\begin{proof}
Otherwise, $w$ has neighbours $y\in U_1$, $y''\in Y\setminus  V(K)$ and a non-neighbour $t \in N(U_2)\cap N(C)$, and there is a vertex $y'\in N(y)\cap U_2$ such that $y't\in E(G)$. By Claim~\ref{lem:edgesinY_A}, if $y\in N(T'_i\cup S_i)$, then
$N(y')\setminus Y\subseteq T'_{i-1}\cup S_{i-1}\cup T'_{i+1}\cup S_{i+1}$.
Assume $t\in T'_{i+1}\cup S_{i+1}$ (the other case is symmetric). Then
$y'', r, y, y'$, $t, c_{i+5}, c_{i+6}, \ldots, c_{i-2}, c_{i-1}$
 is an induced path  of length $t$, a contradiction.
\end{proof}

We are now ready for the final result of this subsection, which resumes all important properties we will need later on.

\begin{lemma}\label{propoX}
There is a set $X\subseteq N(C)$ with $|X|\leq 3t$ such that  for every component~$K$ of $G[Y]$ at least one of the following holds:
\begin{enumerate}[(I)]
\item $V(K)\subseteq N(X)$;\label{good}
\item $K$ is trivial and there is a vertex $c$ on $C$ such that $N(K)\subseteq N(c)$; or\label{trivi}
\item $K$ is bipartite with partition classes $U_1, U_2$, and 
 for each $j=1,2$, and  each $w\in N(U_j)\cap N(C)$ having no neighbours in $N(U_{3-j})\cap N(C)$, we have
\begin{enumerate}[(A)]
\item $w\notin X$; and
\item each $z\in N(w)\setminus U_j$ is adjacent to some neighbour of $N(U_{3-j})$ on $C$.\label{adjac}
\end{enumerate}
Moreover, there are consecutive $c, c', c''\in V(C)$ on $C$ such that 
one of the following holds:
 \setcounter{enumi}{2}
 \begin{enumerate}
 \item[(C)] %There are consecutive $c, c'\in V(C)$ such that 
 $N(U_1)\subseteq U_2\cup N(c)$ and $N(U_2)\subseteq U_1\cup N(c')$; or\label{reduciXX}
\item[(D)]  $N(U_1)\subseteq U_2\cup N(c)\cup N(c'')$ and $N(U_2)\subseteq U_1\cup N(c')$, and furthermore, $$N(U_1)\cap X\cap N(c)\neq\emptyset\neq N(U_1)\cap X\cap N(c'').$$\label{list2XX}
\end{enumerate}
\end{enumerate}
\end{lemma}
\begin{proof}
We let $X':=M\cup B$, where  $M$ is the set from Claim~\ref{coroM} and $B$ is the set from Claim~\ref{lem:colourW}. 
Now, for each  $x\in X'$
we check whether there is a bipartite component $K$  of $G[Y]$, with partition classes $U_1, U_2$, such that $N(x)\cap Y\subseteq U_j$, for some $j\in\{1,2\}$. For any such $x$ and $K$, we check whether $U_{3-j}$ has a neighbour in $X'$. If this is not the case, we choose an arbitrary vertex from $N(U_{3-j})\setminus Y$.
 Add all these vertices to~$X'$, which gives us the set $X$. Note that  for all $x\in X$ and any  bipartite component $K$ of $G[Y]$, 
\begin{equation} 
 \label{ijcombi}
\text{if $N(x)\cap Y\subseteq V(K)$ then $V(K)\subseteq N(X)$.}
\end{equation}
Indeed, to see this note that if one of the newly added vertices fits the role of $x$ in the condition of (\ref{ijcombi}), then the corresponding component $K$ has been used for defining $X$ (coming from the other side). 
Furthermore, the size of $X$ is bounded as desired, and by Claim~\ref{coroM}~(i) and Claim~\ref{lem:colourW}, 
\begin{equation}\label{XXX} 
N(X)\supseteq W\cup \big(Y\cap \bigcup_{i\in [t-1]} N(T'_i)\cap N(T'_{i+2})\big).
\end{equation}

Now, consider a trivial component $K=\{y\}$ of $G[Y]$. By~\eqref{XXX}, if there is an $i\in [t-1]$ such that $y\in N(T'_i)\cap N(T'_{i+2})$, then  $y\in N(X)$ and hence~\eqref{good} holds for $K$. Otherwise, by  Claim~\ref{co}~\eqref{coroindexi} and~\eqref{coroTheNeighbour}, we know that~\eqref{trivi} holds for $K$. 

Let us now turn to the non-trivial components of $G[Y]$. 
By Claim~\ref{Ybipsamenbhds}, any such component $K$ is bipartite with partition classes $U_1, U_2$, and the vertices from each $U_i$ have identical neighbourhoods in $N(C)$. In particular,  $K$ either is contained in $G[W]$ or does not meet $W$ at all. For the former type of components $K$, \eqref{good} holds because of~\eqref{XXX}, so let us assume that $V(K)\cap W=\emptyset$. Then  by Claim~\ref{nononono}, by Claim~\ref{corooo}~\eqref{itemone},  and by~\eqref{ijcombi}, (A) and (B) hold.

If there are no $j\in\{1,2\}$ and $i\in[t-1]$ such that $U_j\subseteq N(T'_i)\cap N(T'_{i+2})$, then by Claim~\ref{corooo}~\eqref{itemtwo}, we have~(C). 
So assume there are $j,i$ such that $U_j\subseteq N(T'_i)\cap N(T'_{i+2})$. Then 
by Claim~\ref{corooo}~(c) there are three vertices $c, c', c''$ as desired, and
by Claim~\ref{COMP} and since $M$ meets each non-empty~$T'_i$ (this is guaranteed by Claim~\ref{coroM}~\eqref{allT'i}), it follows that (D) holds for $K$.
\end{proof}

\subsection{The proof of Lemma~\ref{lem:magic} for $G\in\mathcal G^*$}\label{sec:magic}
This section is devoted to the proof of Lemma~\ref{lem:magic} for all  $G\in\mathcal G^*$. We will use  Lemmas~\ref{propdef:c2d+1} and~\ref{propoX}. \\

\indent {\bf Case 1:}
 $G$ is $C_{t}$-free.\\
In this case, $G$ has an induced cycle $C$ of length $t-2$. We apply Lemma~\ref{propdef:c2d+1} to obtain a set $X$. We let~$\mathcal L$ be the set of all feasible palettes obtained by first precolouring $V(C)\cup X$ (in all possible ways) and then updating.
Now, given a palette $L\in\mathcal L$, let $V_3=V_3(G,L)$ be as in the lemma, that is,
$V_3=\{ v\in V(G) : |L(v)|= 3 \}$. Then since $V(C)$ is coloured,  $V_3\subseteq Y$. 

Let $K'$ be a component of $G[V_3]$. Then $K'$ must be a subgraph of a component $K$ of $G[Y]$, as in Lemma~\ref{propdef:c2d+1} (II) or (III). If $K$ is as in (II), then $K'=K=\{y\}$ for some $y\in Y$, and there is a vertex $c\in V(C)$ dominating $y$. Thus, the colour of $c$ is missing on the list of each neighbour of $y$, implying that $y$ is a reducible vertex, which is as desired.

So assume $K$ is as in (III) of Lemma~\ref{propdef:c2d+1}, with bipartition classes $U_1$, $U_2$.  If $K'$ is non-trivial, then, since $L$ is updated, we know that each vertex in $N(K')\cap N(C)$ has a list of size~$2$. More precisely, by (III)(A), for $j=1,2$, there is a colour~$\alpha_j$ missing on the lists of each vertex $v\in N(U_j)\cap N(C)$, and all neighbours of $v$ on $C$ must have colour $\alpha_j$. By property~(III)(B), it is clear that $\alpha_1\neq \alpha_2$. Therefore $K'$ is reducible, which is as desired. 

It remains to treat the case that $K$ is as in  Lemma~\ref{propdef:c2d+1}~(III) and $K'=\{y\}$ is trivial. Assume $y\in U_1$ (the other case is symmetric). Then, since $L$ is updated, and $|L(y)|=3$, the vertices in $U_2$ have lists of size $2$, and there must be a vertex $z\in N(U_2)\cap N(C)$ having list size $1$. Say $z$ is coloured $\alpha_0$. By~(III)(C), vertex $z$ is complete to $N(y)\cap N(C)$, meaning that $\alpha_0$ is missing on the lists of all neighbours of $y$. So $y$ is reducible, as desired. 

\smallskip
{\bf Case 2:}
 $G$ has an induced cycle $C$ of length $C_{t}$.\\
We apply Lemma~\ref{propoX}. Let $\mathcal L$ be the set of all feasible palettes obtained by  precolouring $V(C)\cup X$ in all possible ways and then updating.
Given a palette $L\in\mathcal L$, we consider
$V_3=\{ v\in V(G) : |L(v)|= 3 \}$. Since $V(C)$ is coloured, $V_3\subseteq Y$. 
Consider any component  $K'$ of $G[V_3]$. Then $K'$ must be a subgraph of a component $K$ of $G[Y]$ as in~\eqref{trivi}  or~\eqref{reduciXX} of Lemma~\ref{propoX}. 
If $K$ is as in~\eqref{trivi}, then there is a vertex $c$ on $C$ such that $N(K)\subseteq N(c)$. So, since $L$ is an updated palette, the colour of~$c$ is missing in the list of every neighbour of~$K$. Hence $K$ is reducible, which is as desired.

So we can assume  $K$ is as in~\eqref{reduciXX},  with partition classes $U_1, U_2$.
Observe that  vertices from $Y$ can only have lost colours from their list by updating. So, by Claim~\ref{Ybipsamenbhds}, if $K'$ is non trivial, then $K=K'$. Moreover,  if $K'=\{y\}$ is trivial, then $y\in U_j$ for some $j\in\{1,2\}$, and all vertices in $U_{3-j}$ have list size 2 (they cannot have list size 1, since $|L(y)|=3$).

If $K$ satifies~\eqref{reduciXX}(C),
we
consider the  vertices $c, c'$. Being adjacent, they must have been assigned distinct colours $\alpha_1, \alpha_2$. Since the palette $L$ is updated, $\alpha_1$ is missing in the lists of all neighbours of $U_1$ in $N(C)$, and $\alpha_2$ is missing in the lists of all neighbours of $U_2$ in $N(C)$. 
 Therefore, if $K'$ is non-trivial, then, since the list of each vertex in $K'=K$ contains all three colours, $K'$ is  reducible, which is as desired. 
 
 So, we can assume $K'=\{y\}$ is trivial. Assume $y\in U_1$ (the other case is symmetric). Note that 
 \begin{equation}\label{list2}
 \text{every neighbour of $N(U_1)$ on $C$ is coloured $\alpha_1$.}
 \end{equation}
Let $\alpha_0$ be the colour missing in the lists of the vertices from $U_{2}$. If $\alpha_0=\alpha_1$, then  $y$ is reducible, which is  as desired. We claim that this is the case.
  So assume otherwise, that is, assume
 $\alpha_0\neq \alpha_1$. There must be a vertex $w$ in $N(U_{2})\cap N(C)$ coloured~$\alpha_0$. Note that $w$ is not adjacent to any vertex $v\in N(y)\cap N(C)$, since $v$ then would miss two colours on its list, which is impossible (as $|L(y)|=3$). 
 Now,~\eqref{list2XX}(A) implies that $w\notin X$, and therefore,  $w$ has a neighbour $z$ that is coloured~$\alpha_1$ (as there is no other possible reason for $w$ to be coloured $\alpha_0$).
 But then, by~\eqref{list2XX}(B), and by~\eqref{list2}, $z$ is adjacent to a vertex coloured $\alpha_{1}$. Since $L$ is feasible and updated, we arrive at a contradiction.

If $K$ satisfies~\eqref{list2XX}(D), we
consider the  vertices $c, c', c''$. Belonging to $V(C)$, they must have been coloured in $L$, say they were assigned colours $\alpha_1, \alpha_2, \alpha_3$. Because of the adjacencies of $c, c', c''$, we have  $\alpha_1\neq \alpha_2\neq \alpha_3$. If $\alpha_1=\alpha_3$, then we can argue exactly as in the previous case that $K'$ is a bipartite reducible subgraph of $G$, or that $K'=G[y]$ is trivial and $y$ reducible, which is as desired. So, we assume  $\alpha_1\neq \alpha_3$.

Let $x_1\in N(U_1)\cap X\cap N(c)$, and let $x_2\in N(U_1)\cap X\cap N(c'')$.  If $x_1$ and $x_2$ have been assigned different colours, then the vertices in $U_1$ have list size 1, and therefore the vertices in $U_2$ have list size at most 2, a contradiction. So we can assume $x_1$ and $x_2$ have been assigned the same colour. Since colour $\alpha_1$ is missing in the list of $x_1$, colour $\alpha_3$ is missing in the list of $x_2$, and $L$ is feasible, that colour has to be $\alpha_2$. This means $\alpha_2$ is missing in the list of all vertices in $U_1$. Thus $K'=\{y\}$ has to be trivial, with $y\in U_2$. Furthermore, $y$ is reducible (for colour $\alpha_2$), which is as desired.

\section{Adjustments of the proof of Lemma~\ref{lem:magic} for $G\notin \mathcal G^*$}\label{sec:C8}

In this section, we modify the proof of the previous section  to a proof of Lemma~\ref{lem:magic} for all $G$. 

\subsection{The strategy}
No modifications have to be made if $G$ has an induced $C_{t}$, so we only need to focus on the case treated in Subsection~\ref{sec:shortcycle}. Because of the argument there, we can now assume that $t>9$.
We let $C=c_0, c_1, ..., c_{t-3}, c_0$ be an induced cycle in $G$, and let  $D_i$, $T_i$ and $Y$ be the sets from Claim~\ref{basecyc} for~$C$.
As noted in Subsection~\ref{sec:shortcycle}, the only use of the fact that the graph~$G$ was supposed to be $C_8$-free if $t>9$ was in the proof of Claim~\ref{lem:y_i^*}.
It is easy to see that in order to apply to arbitrary~$G$, Claim~\ref{lem:y_i^*} has to be rewritten as follows.

\begin{claim}
\label{newlem:y_i^*}
For each $y\in Y$ there is an $i\in[t-3]$ such that either \begin{enumerate}[(a)] \item $N(y)\setminus Y\subseteq D_{j}\cup  T_{i}\cup T_{i+2}$ for some $j\in \{i-2, i, i+2, i+4\}$; or
\item $N(y)\subseteq D_{i}\cup T_{i}\cup T_{i+2}\cup D_{i+4}$ and $N(y)\cap D_i\neq\emptyset\neq N(y)\cap D_{i+4}$.
\end{enumerate}
\end{claim}

This leads to the following modification in the statement of Lemma~\ref{propdef:c2d+1}.  Property (II) has to be replaced with the following property, where we let $Y^*$ denote the set of all vertices of $Y$ that are as in  Claim~\ref{newlem:y_i^*}(b) for some $i\in [t-3]$.
\begin{enumerate}[(I')]\setcounter{enumi}{1}
\item $K=\{y\}$ is trivial and one of the following holds:
\begin{enumerate}[(a)]
\item there is a vertex $c\in V(C)$ such that $N(K)\subseteq N(c)$; or
\item $y\in Y^*$.
\end{enumerate}
\end{enumerate}

In order to deal with the vertices in $Y^*$, we will 
 colour them, their neighbours, and some of their second neighbours, by assigning colours synchronously to whole sets, in a similar way as we colour bipartite reducible components. 
For this, consider $(G,L)$,  and an updated feasible subpalette $L'$ of~$L$. If there is $S\subseteq V(G)$ such that $|L'(v)|=1$ for each $v\in S$ and $L(v)=L'(v)$ for each $v\in V(G)\setminus S$, then we say $L'$ is a {\it nice} reduction of $L$.
It easy to see that the following lemma holds.

\begin{lemma}\label{lem:reduction}
Given two palettes $L$ and $L'$ of a graph $G$ such that $L'$ is a nice reduction of $L$, we have that $(G,L)$ is 3-colourable if and only if $(G,L')$ is 3-colourable.
\end{lemma}

\subsection{Preliminaries for generating the set of palettes $\mathcal L$}\label{sec:palettes}

We first need a little more structural analysis. We start with possible edges inside $N(C)$.

%SAME CLAIM WITH OPTION u in T
%\begin{claim}\label{lem:edgesinN(C)_B}
%For each $uv\in E(G[N(C)])$ and each $i\in[2d]$ the following holds.
%\begin{enumerate}[(a)]
%\item If $u\in D_i$ then $v\in D_{i+3}\cup \bigcup_{j\in\{i-3, i-1, i+1\}}\{D_{j}\cup  T_{j}\}$; and
%\item if $u\in T_i$ then $v\in D_{i+3}\cup \bigcup_{j\in\{i-1, i+1\}}\{D_{j}\cup  T_{j}\}$.
%\end{enumerate}
%\end{claim}
%\begin{proof}
%Any case not covered by~(a) or~(b) would lead to an  induced odd cycle of length $\le 2d-1$ on vertices from $V(C)\cup \{ u,v\}$. But such a cycle is forbidden in $G$. 
%\end{proof}
\begin{claim}\label{lem:edgesinN(C)_B}
Let $uv\in E(G[N(C)])$ and let $i\in[t-3]$ such that $u\in D_i$.
Then $$v\in D_{i+3}\cup \bigcup_{j\in\{i-3, i-1, i+1\}}\{D_{j}\cup  T_{j}\}.$$\end{claim}
\begin{proof}
Vertex $v$ being in any other set would lead to an  induced odd cycle of length $\le t-4$ on vertices from $V(C)\cup \{ u,v\}$. But such a cycle is forbidden in $G$. 
\end{proof}

% LEMMA COMPLETENESS TO Ti-1 and Ti+3 and perhaps tp T_i+1
%\begin{lemma}\label{lem:Ti-1andi+1}
%Let $y\in Y^*$, and let $i\in [2d]$ be such that both $\tilde D_i:=N(y)\cap D_i$ and $\tilde D_{i+4}:=N(y)\cap  D_{i+4}$ are nonempty. Then 
%\begin{enumerate}[(a)]
%\item $T_{i-1}$ is complete to $\tilde D_i$;
%\item $T_{i+3}$ is complete to $\tilde D_{i+4}$; and
%\item either
%$T_{i+1}$ is complete to $\tilde D_i\cup \tilde D_{i+4}$ or there is no edge between $T_{i+1}$ and $\tilde D_i\cup \tilde D_{i+4}$.
%\end{enumerate}
%\end{lemma}
%\begin{proof}
%By Lemma~\ref{lem:edgesinN(C)_B}, there are no edges between $T_{i-1}$ and $\tilde D_{i+4}$, or between $T_{i+3}$ and $\tilde D_i$. Now 
%suppose there are two non-adjacent vertices $t_{i-1}\in T_{i-1}$, $d_i\in \tilde D_i$. By Lemma~\ref{newlem:y_i^*}~(b), $t_{i-1}$ is not adajcent to $y$. So the path $d_i, y, d_{i+4}, c_{i+4}$, $c_{i+5}$, $...$, $c_{i-1}, t_{i-1}, c_{i+1}, c_{i+2}$ is induced and has length $2d+3$, a contradiction. This proves that~(a) holds, and (b) follows symmetrically.
%
%It remains to show (c). For this, assume there is a vertex $t_{i+1}\in T_{i+1}$ having a neighbour $d_{i}\in\tilde D_i$ and a non-neighbour $d_{i+4}\in\tilde D_{i+4}$. Since $G$ is $C_3$-free, $t_{i+1}$ is not adajcent to $y$. So, $c_{i+2}$, $c_{i+1}$, $t_{i+1}$, $d_i$, $y$, $d_{i+4}$, $c_{i+4}$, $c_{i+5}$, $...$, $c_{i-1}$ is an induced path of length $2d+3$, a contradiction. A symmetric argument can be used if we swap the roles of $d_i$ and $d_{i+4}$. Hence (c) follows.
%\end{proof}

We now explore the first and second neighbourhoods of vertices from $Y^*$. For each $i\in[t-3]$, let  $Y^*_i\subseteq Y^*$ denote the set of all vertices from $Y$ that are adjacent to both $D_i$ and $D_{i+4}$. 

\begin{claim}\label{lem:terrible}
Let $i\in [t-3]$, and let $y\in Y^*_i$. Then, setting $D^*_{i}:=N(y)\cap  D_{i}$ and $D^*_{i+4}:=N(y)\cap  D_{i+4}$,
\begin{enumerate}[(a)]
\item \label{Di-2}
$D_{i-2}=D_{i+2}=D_{i+6}=\emptyset$;
\item \label{b} $N(D^*_i)\subseteq V(C)\cup D_{i+3}\cup T_{i-1}\cup T_{i+1}\cup Y$; and%, where $D^*_i:=N(y)\cap D_i$;
\item \label{c} $N(D^*_{i+4})\subseteq V(C)\cup D_{i+1}\cup T_{i+1}\cup T_{i+3}\cup Y$.%, where $D^*_{i+4}:=N(y)\cap  D_{i+4}$,
\end{enumerate}
Moreover, letting $D^{+}_{j}$, for $j\in\{i, i+1, i+3, i+4\}$, denote the set of all vertices from $D_j$ belonging~to components of $G^+_i:=G[D_i\cup D_{i+1}\cup D_{i+3}\cup D_{i+4}]$ that contain a vertex from $D^*_i\cup D^*_{i+4}$, we have
\begin{enumerate}[(a)]\setcounter{enumi}{3}
\item $N(D^+_{i+1})\setminus (D^+_{i}\cup D^+_{i+4})\subseteq V(C)\cup  T_{i-2}\cup T_{i}\cup  T_{i+2}\cup Y$; \label{d}
\item $N(D^+_{i+3})\setminus (D^+_{i}\cup D^+_{i+4})\subseteq V(C)\cup   T_{i}\cup  T_{i+2}\cup  T_{i+4}\cup Y$; and\label{e}
\item $N(D^{+}_{i}\cup D^+_{i+4})\setminus (D^+_{i+1}\cup D^+_{i+3})\subseteq V(C)\cup  T_{i-1}\cup T_{i+1}\cup T_{i+3}\cup %Y^*_i\cup 
(Y\cap N(D^*_i\cup D^*_{i+4}))$.\label{f}
\end{enumerate}
\end{claim}

\begin{proof}
By Claim~\ref{newlem:y_i^*}, we know that $y$ only has neighbours in $D^*_i\cup T_{i}\cup  T_{i+2}\cup D^*_{i+4}$. This fact will be implicitly used below for seeing that some of the paths we present are induced.

In order to see~\eqref{Di-2}, note that
any vertex $d_{i-2}\in D_{i-2}$ would lead to the induced path $y, d^*_i, c_{i}, c_{i+1}$, $c_{i+2}, ..., c_{i-2}, d_{i-2}$, where $d^*_i\in D^*_i$. This path has length $t$, which is impossible. Using symmetric arguments, we see  that $D_{i+2}$ and $D_{i+6}$ are empty, too. This proves~\eqref{Di-2}.

Furthermore, note that
\begin{equation}\label{didi+1}
\text{there are no edges between $D^*_i$ and $D_{i-1}\cup D_{i+1}$.}
\end{equation}
Indeed, any such edge, say $d^*_id_{i+1}$, with $d^*_i\in D^*_i$ and $d_{i+1}\in D_{i+1}$, gives rise to the path $c_{i-1}, c_{i-2}, ...$, $c_{i+2}, c_{i+1}, d_{i+1}, d^*_{i}, y$. This path is induced and has length $t$, which is forbidden. 
In the same way, we can show that
\begin{equation}\label{didi4+1}
\text{there are no edges between $D^*_{i+4}$ and $D_{i+3}\cup D_{i+5}$.}
\end{equation}

Next, in order to see (\ref{b}),  consider an edge from $d^*_i\in D^*_i$ to $x\in N(C)\setminus (D_{i+3}\cup T_{i-1}\cup T_{i+1})$. By Claim~\ref{lem:edgesinN(C)_B} and because of~\eqref{didi+1}, we have $x\in D_{i-3}\cup T_{i-3}$ and $xd_{i+4}\notin E(G)$. So the cycle $x, d^*_i, y, d_{i+4}, c_{i+4}, c_{i+5}, ..., c_{i-3}, x$ is induced and of length $t-4$. 
This proves~(\ref{b}), and a symmetric argument shows~(\ref{c}).
Also, 
observe that (\ref{d}) and (\ref{e}) follow directly from  Claim~\ref{lem:edgesinN(C)_B} and from~\eqref{Di-2}. 

So it only remains to prove (\ref{f}). For this, consider a possible edge from a vertex $d^+_i\in D^{+}_i\cup D^+_{i+4}$ to a vertex  $x\notin V(C)\cup D^+_{i+1}\cup D^+_{i+3}\cup T_{i-1}\cup T_{i+1}\cup T_{i+3}\cup  %Y^*_i\cup 
(Y\cap N(D^*_i\cup D^*_{i+4}))$. Take a  shortest path $P=d^+_i, w_1, w_2, ..., w_m$  from $d^+_i$ to $D^*_i\cup D^*_{i+4}$ in $G^+_i$. We may assume that the only neighbour of $x$ in $V(P)\cap (D^+_i\cup D^{+}_{i+4})$ is~$d^+_i$ (after possibly replacing the vertex playing the role of $d^+_i$). So, by Claim~\ref{lem:edgesinN(C)_B},  and by Claim~\ref{newlem:y_i^*}, and since we assume that $t>9$,
\begin{equation}\label{nonbsonpath}
\text{$x$ has no neighbours in $V(P)\setminus\{d^+_i\}$,}
\end{equation}
Also, observe that  because of~\eqref{b} and~\eqref{c}, we know that $x\notin N(D^*_i\cup D^*_{i+4})$,  and thus, $d^+_i\notin D^*_i\cup D^*_{i+4}$. So, $w_1\in D^+_{i+1}\cup D^+_{i+3}$ and $m\ge 2$.

Because of symmetry, we may assume that $d^+_i\in D^{+}_i$.
By Claim~\ref{lem:edgesinN(C)_B}, 
$$x\in D_{i-3}\cup T_{i-3}\cup D_{i-1}\cup Y\setminus N(D^*_i\cup D^*_{i+4}).$$

First we treat the case  $x\in D_{i-3}\cup T_{i-3}$. 
If $w_1\in D^+_{i+3}$, then the induced cycle $x, d^{+}_i, w_1, c_{i+3}$, $c_{i+4}, ..., c_{i-3}, x$ has length $t-4$, which is forbidden. So $w_1\in D^+_{i+1}$. 
If $w_2\in D_{i+4}$, we obtain the forbidden induced cycle $x, d^{+}_i, w_1, w_2, c_{i+4}, ..., c_{i-3}, x$ of length $t-4$. So we can assume $w_2\notin D_{i+4}$, and hence, $w_2\in D_i$. In particular, by~\eqref{didi+1}, $w_2\in D^+_i\setminus D^*_i$, and so, by the choice of $P$, $w_1$ is not adjacent to $D^*_i\cup D^*_{i+4}$.
Consider the path $c_{i+2}, c_{i+1}, w_1, d^+_i, x, c_{i-3}, c_{i-4}, ..., c_{i+4}, d^*_{i+4}, y, d^*_i$, where $d^*_j\in N(y)\cap D^*_j$ for $j=i, i+4$. This path has length $t$, and it is  induced, as we already saw that the edges $d^*_{i+4}w_1$, $d^*_iw_1$ are not present, and by~\eqref{b} and~\eqref{c}, the edges $d^*_ix$, $d^*_{i+4}x$ are not present. We arrived at a contradiction. This proves that $x\notin D_{i-3}\cup T_{i-3}$. 

Now, let us consider the possibility that $x\in D_{i-1}$. First assume $w_m\in D^*_i$. Then by~\eqref{didi+1}, vertex $w_{m-1}$ lies in $D^+_{i+3}$ and no vertex in $\{w_1, ..., w_{m}\}$ is adjacent to $D^*_{i+4}$.  Then, by the choice of $P$, and by~\eqref{nonbsonpath}, we know that
 $x, d^{+}_i, w_1, w_2, ..., w_m, y, d^*_{i+4},  c_{i+4}, c_{i+5},  ...$, $c_{i-1}, x$ (where $d^*_{i+4}\in D^*_{i+4}$)
is an induced cycle of length at least $t$,  a contradiction. 
  So we can assume $w_m\in D^*_{i+4}$, which, by~\eqref{didi4+1}, implies that $w_{m-1}$ lies in $D^+_{i+1}$. % and is not adjacent to $D^*_{i}$.
 Because of the induced cycle $x, d^{+}_i, w_1, w_2, ..., w_m,  c_{i+4}, c_{i+5},  ..., c_{i-1}, x$, which has length $t-4+m$, we are done unless $m<4$. So assume that $m<4$, and note that then $m=2$, by Claim~\ref{lem:edgesinN(C)_B}.
Consider the cycle $x, d^{+}_i, w_1, c_{i+1}, c_{i+2},  ..., c_{i-1}, x$, which is induced and has length $t$, a contradiction. We conclude that $x\notin D_{i-1}$.

It remains to eliminate the case that $x\in Y\setminus N(D^*_i\cup D^*_{i+4})$. In this case, consider the path $x, d^{+}_i, w_1, w_2, ..., w_m, y, d^*_{i+4},  c_{i+4}, c_{i+5},  ..., c_{i-1}$  (where $d^*_{i+4}\in D^*_{i+4}$) if $w_m\in D^*_i$, and the path $x, d^{+}_i$, $w_1, w_2, ..., w_m, c_{i+4}, c_{i+5},  ..., c_{i-1}$ if $w_m\in D^*_{i+4}$ and $m\ge 4$. By~\eqref{nonbsonpath}, and as $x\notin  N(D_i^*\cup D^*_{i+4})$,
both these paths are induced and have length at least $t$, which is forbidden. So we can assume that $w_m\in D^*_{i+4}$ and $m<4$, implying that $m=2$ and $w_1\in D^+_{i+1}$. Consider the induced path
$x, d^{+}_i$, $w_1, c_{i+1}, c_{i+2},  ..., c_{i-1}$ of length $t$, a contradiction. Hence $x\notin Y\setminus  N(D^*_i\cup D^*_{i+4})$, as desired.
\end{proof}

\begin{claim}\label{lem:terribleterribleY}
Let $y\in Y^*_i$,  and let $D^*_i$, $D^*_{i+4}$, $D_{i+3}^+$ and $D_{i+1}^+$ be as in Claim~\ref{lem:terrible}. Let $y'\in Y$.
\begin{enumerate}[(a)]
\item If  $N(y')\cap D^+_{i+1}\neq\emptyset$, then either $N(y')\setminus D_{i+1}^+\subseteq T_{i-1}\cup T_{i+1}$ or $N(y')\setminus D_{i+1}^+\subseteq T_{i+1}\cup T_{i+3}$.
\item If $N(y')\cap D^+_{i+3}\neq\emptyset$, then either $N(y')\setminus D_{i+3}^+\subseteq T_{i-1}\cup T_{i+1}$ or $N(y')\setminus D_{i+3}^+\subseteq T_{i+1}\cup T_{i+3}$.
\end{enumerate}
\end{claim}
\begin{proof}
We only prove (b), since (a) is symmetric. Let $d^+_{i+3}\in N(y')\cap D^+_{i+3}$, and let  $P=d^+_{i+3}, w_1$, $w_2, ..., w_m$  be a shortest path from $d^+_{i+3}$ to $D^*_i\cup D^*_{i+4}$ in $G^+_i$ (where $G^+_i$ is as in Claim~\ref{lem:terrible}). Because of Claim~\ref{newlem:y_i^*}, and since we assume $t>9$, we may assume that $y'$ has no neighbours in $V(P)\setminus\{d^+_{i+3}\}$ (after possibly changing the vertex in $D^+_{i+3}$ playing the role of $d^+_{i+3}$). Also, note that the only neighbour of $y$ in $V(P)$ is $w_m$, 
and that
\begin{equation}\label{novertexw}
\text{no vertex in $\{w_1, ..., w_{m-2}\}$ is adjacent to $D^*_i\cup D^*_{i+4}$.}
\end{equation}

 We claim that for all $\ell<m$,
\begin{equation}\label{wisgood}
w_\ell\notin D^+_{i}\cup D^+_{i+4}.
\end{equation}
Indeed, if this is not true, then let $\ell<m$ be the smallest index such that $w_\ell\in D^+_{i}\cup D^+_{i+4}$. Let $d^*_j\in D^*_j$, for $j=i, i+4$. If $w_\ell\in D^+_i$, then consider the path
$y', d^+_{i+3}, w_{1}, w_2, ..., w_\ell, c_{i}, c_{i-1}, ..., c_{i+5}, c_{i+4}$, $d^*_{i+4}, y$, which, by~\eqref{novertexw}, is induced and has length at least $t$. If $w_\ell\in D^+_{i+4}$, then consider the path 
$y', d^+_{i+3}, w_1, w_2, ..., w_\ell, c_{i+4}, c_{i+5}, ..., c_{i}, d^*_i, y$, which has length $t$ and by~\eqref{novertexw}, is induced. As such paths are forbidden, we proved~\eqref{wisgood}.

Note that~\eqref{wisgood} together with Claim~\ref{lem:edgesinN(C)_B} implies that $m=1$ and $$w_1\in D^*_i.$$ 

We will now show that 
\begin{equation}\label{noNBs}
\text{$y'$ has no neighbours in $D_{i-1} \cup (D_{i+3}\setminus D_{i+3}^+)$.}
\end{equation}
For contradiction, suppose $y'$ has a neighbour  $x\in D_{i-1}\cup (D_{i+3}\setminus D_{i+3}^+)$. Then  $x, y', d^+_{i+3}, w_1, y, d^*_{i+4}$, $c_{i+4}, c_{i+5}, ..., c_{i-1}$ is a path, where $d^*_{i+4}\in D^*_{i+4}$. This path has length $t$, and because of Claim~\ref{lem:edgesinN(C)_B} and Claim~\ref{lem:terrible}~\eqref{b}, the only possible edge is $xc_{i-1}$, leading to an induced cycle of the same length, a contradiction. 
This proves~\eqref{noNBs}.

Next, we show that
\begin{equation}\label{noNBs2}
\text{$y'$ has no neighbours in $T_{i+5}\cup D_{i+7}$.}
\end{equation}

In order to see~\eqref{noNBs2}, assume $y'$ has a neighbour
 $x\in T_{i+5}\cup D_{i+7}$. Consider the cycle $x, y', d^+_{i+3}, w_1$, $c_i, c_{i-1}, ..., c_{i+7}, x$. By Claim~\ref{lem:terrible}(\ref{b}), this cycle is induced, and furthermore, it has length $t-4$. This is a contradiction, which proves~\eqref{noNBs2}.

By~\eqref{noNBs} and~\eqref{noNBs2}, and
because of Claims~\ref{DY...} and~\ref{newlem:y_i^*}, it only remains to show that $y'$ cannot have neighbours in both $T_{i-1}$ and $T_{i+3}$. Suppose otherwise, and let $t_{j}\in N(y')\cap T_{j}$, for $j=i-1, i+3$. Then $c_{i-1}, t_{i-1}, y', t_{i+3}, c_{i+3}$, $c_{i+4}, ..., c_{i-1}$ is an induced cycle of length $t-4$, a contradiction.
\end{proof}

\begin{claim}\label{lem:terribleY}
Let $y\in Y^*_i$, $y'\in Y\setminus Y^*_i$ such that $y'$ has a neighbour in $D^*:=N(y)\cap  (D_i\cup D_{i+4})$. Then \begin{enumerate}[(a)]
\item $N(y')\subseteq D_i\cup T_i\cup T_{i+2}\cup D_{i+4}$; and
\item if $d\in D\cap N(y')\setminus N(y)$, then $N(d)\subseteq N(D^*)$.
\end{enumerate}
\end{claim}

\begin{proof}
Because of symmetry, we can assume that $y'\in N(D_i)$. Let $d_i\in N(y)\cap N(y')\cap D_i$ and let $d_{i+4}\in N(y)\cap  D_{i+4}$. 
Note that by Claim~\ref{lem:terrible}~\eqref{b}, $y$ has no neighbours in $D_{i-4}\cup T_{i-4}\cup T_{i-2}$. 
Also, note that by Claim~\ref{DY...}, $y'$ has no neighbours in $Y$, and as $y'\notin  Y^*_i$, we know that $y'd_{i+4}\notin E(G)$.

We first show (a).  If $y'$ has a neighbour $t_{i-4}\in D_{i-4}\cup T_{i-4}$, then $t_{i-4}, y', d_i, y, d_{i+4}, c_{i+4}, c_{i+5}, ...$, $c_{i-4}, t_{i-4}$ is an induced cycle of length $t-4$, a contradiction. If $y'$ has a neighbour $t_{i-2}\in T_{i-2}$, then $c_{i+6}, c_{i+7}, ..., c_{i-2}, t_{i-2}, y', d_i, y, d_{i+4}, c_{i+4}, c_{i+3}, c_{i+2}, c_{i+1}$ is an induced path of length $t$, a contradiction. So by Claim~\ref{newlem:y_i^*}, item (a) follows.

Let us now show (b). For this, let $d\in D\cap N(y')\setminus N(y)$ and note that as $y'\in N(D_i)$ and $y'\notin Y^*_i$ we have $d\in D_i$. Assume $xd$ is an edge, with $x\in V(G)\setminus N(D^*)$. Consider the path $x, d, y', d_i, y, d_{i+4}, c_{i+4}, c_{i+5}, ...,c_{i-1}$, which has length $t$, and therefore cannot be induced. The only possible chords are of the form $xc_j$, for $j=i+5, i+6, ..., i-1$. Morover, if the only chord is $xc_{i-1}$, we obtain an induced cycle of length $t$, which is forbidden. Therefore, and because of Claim~\ref{lem:edgesinN(C)_B}, we know that $x\in D_{i-3}\cup T_{i-3}$. Then, consider the path $c_{i+1}, c_{i+2}, c_{i+3}$, $c_{i+4}, d_{i+4}, y, d_i, y', d, x, c_{i-3}, c_{i-4}, ..., c_{i+6}$. This path is induced and has length $t$, a contradiction.
\end{proof}

\subsection{The proof of Lemma~\ref{newlem:magic} for all $G$}\label{sec:newmagic}

%\begin{proof}[Proof of Lemma~\ref{newlem:magic}]
Given a $(P_{t}, \Codd)$-free graph $G$, we can assume $G$ is $C_{t}$-free, and so, $G$ has an induced cycle  $C$ of length $t-2$. We find, as before, a set $X$ as in Lemma~\ref{propdef:c2d+1}, but we will not colour all of it yet. 
 Instead, for every colouring of $V(C)$, we generate a palette. We take $\mathcal L'$ as the set of all such palettes that are feasible. 
For each $L'\in\mathcal L'$, we proceed as follows.

For each $i\in[t-3]$, we consider the set $Y^*_i\subseteq Y^*$. 
Set \[D_j^{*(i)}:=\bigcup_{y\in Y^*_i}(N(y)\cap D_j)\text{ for }j=i, i+4.\] 
For $j\in\{i, i+1, i+3, i+4\}$, let 
$D^{+(i)}_{j}$ denote the set of all vertices from $D_j$ belonging to components of $$G^{+(i)}_i:=G[D_i\cup D_{i+1}\cup D_{i+3}\cup D_{i+4}]$$ that contain some vertex from $D^{*(i)}_i\cup D^{*(i)}_{i+4}$.
Finally, we let $Z^{*(i)}$ be the set of all $y\in Y\setminus Y^*_i$ that are adjacent to $V(G^{+(i)}_i)$. By
Claim~\ref{lem:terrible}~\eqref{f}, Claim~\ref{lem:terribleterribleY} and Claim~\ref{lem:terribleY}, we know that $Z^{*(i)}$ partitions into three sets $Z^{*(i)}_1$, $Z^{*(i)}_2$ and $Z^{*(i)}_3$, such that  for $j=1,2,3$, 
\begin{equation}\label{NB2}
N(Z^{*(i)}_j)\setminus D\subseteq T_{i+j-2}\cup T_{i+j}.
\end{equation}
Moreover, 
\begin{equation}\label{NB3}
N(Z^{*(i)}_j)\cap D \subseteq D^{+(i)}_{i+1}\cup D^{+(i)}_{i+3}\text{ for $j=1,3$},
\end{equation}
 and $N(Z^{*(i)}_2)\cap D \subseteq D_i\cup D_{i+4}$. Further, for $j=i, i+4$, setting 
$$D^{++(i)}_{j}:= N(Z^{*(i)}_2)\cap (D_j\setminus D^{+(i)}_{j})$$  we have, by Claim~\ref{lem:terribleY} (b), and by Claim~\ref{lem:terrible}~\eqref{b} and~\eqref{c}, 
\begin{equation}\label{NBD++}
\text{$N(D^{++(i)}_{j})\subseteq Y^*_i\cup Z^{*(i)}_2\cup\{c_j\}\cup T$,}
\end{equation}
with 
\begin{equation}\label{NBD+++}
\text{ $N(D^{++(i)}_{i})\cap T\subseteq T_{i-1}\cup T_{i+1}$ and $N(D^{++(i)}_{i+4})\cap T\subseteq T_{i+1}\cup T_{i+3}$.}
\end{equation}

Set $$F(Y^*_i):=Y^*_i\cup  Z^{*(i)}\bigcup_{j=i, i+1, i+3, i+4}D_j^{+(i)}\cup \bigcup_{j=i, i+4}D_j^{++(i)}.$$
By Claim~\ref{newlem:y_i^*}, Claim~\ref{lem:terrible}, Claim~\ref{lem:terribleterribleY} and Claim~\ref{lem:terribleY} and because of~\eqref{NB2},~\eqref{NB3} and~\eqref{NBD++}, we know  that 
$N(F^*_i)\subseteq V(C)\cup T$, and
moreover, $F(Y^*_i)\cap F(Y^*_j)=\emptyset$ for distinct $i,j\in [t-3]$. So we can treat $Y^*_i$ and $Y^*_j$ independently.

For $i\in [t-3]$, let $\alpha_{i}$ be the colour assigned to $c_i$ in $L'$. Colour, for each $i\in [t-3]$ with $Y^*_i\neq \emptyset$, all vertices from the set $F(Y^*_i)$ as follows:

\begin{itemize}
\item Assign colour $\alpha_{i+1}$ to all vertices in $D_{i}^{+(i)}\cup D_{i}^{++(i)}\cup Z^{*(i)}_1$;
\item   assign colour $\alpha_{i+2}$ to all vertices in $D_{i+1}^{+(i)}\cup D_{i+3}^{+(i)}\cup Y^*_i\cup Z^{*(i)}_2$; and 
\item assign colour $\alpha_{i+3}$ to all vertices in $D_{i+4}^{+(i)}\cup D_{i+4}^{++(i)}\cup Z^{*(i)}_3$.
\end{itemize}
After updating, we call the obtained palette $L''$.
Note that
for $j=1,2,3$
we have that $\alpha_{i+j}\notin L'(t)$ for $t\in T_{i+j-2}\cup T_{i+j}$. Hence, by~\eqref{NB2},~\eqref{NB3},~\eqref{NBD++} and~\eqref{NBD+++}, and by Claims~\ref{lem:terrible},~\ref{lem:terribleterribleY} and~\ref{lem:terribleY},  
this colouring is valid, and moreover, the neighbours of $F(Y^*_i)$ have not lost any colours in their lists, for $i\in [t-3]$. So by Lemma~\ref{lem:reduction}, $(G,L'')$ is $3$-colourable if and only if $(G, L')$ is $3$-colourable, which means that we can work with $L''$ instead of $L'$. For each $L''$ obtained in this way,  we colour $X\setminus V(C)$ in all possible ways, update the palette, and, if feasible, we add it to the set $\mathcal L$.

Now, given a palette $L\in\mathcal L$, let $V_3=V_3(G,L)\subseteq Y$ be as in the lemma and let $K'$ be a component of $G[V_3]$. Then $K'$ must be a subgraph of a component $K$ of $G[Y]$, as in Lemma~\ref{propdef:c2d+1}. Note that  $K$ is not as in (II') with $K=\{y\}$ and $y\in Y^*$, since then $|L(y)|=1$, by our previous arguments.  Hence we can proceed exactly as in the case when $G\in \mathcal G^*$. This finishes the proof.
%\end{proof}

\section{The overall complexity of the algorithm}\label{sec:time}

In this section we analyse the complexity of the algorithm. Let $m$, $n$ be the number of edges and vertices of $G$. We assume $t\in\mathbb N$, $t\ge 9$ is  odd.

We start by analysing  the easier case that $G\in\mathcal G^*$, and leave the analysis for the case $G\notin\mathcal G^*$ to the end.

%\subsection{If $G\in\mathcal G^*$}

{\bf Finding the cycle $C$ and the set $X$.}
It can be checked  in $O(m)$ time whether $G$ is bipartite. If it is not, we will find
an induced cycle $C$ that either has length $t-2$ or $t$. If $C$ has length $t-2$, we will work with $C$ as if the graph $G$ was $C_{t}$-free -- should we find, at any point in our process, an induced cycle $C'$ of length $t$, we abort the process, and start afresh with $C'$. 

We check the vertices in  $N(C)$
and partition them into sets $D_i$, $T_i$, $T'_i$, $S_i$. We check the remaining vertices, and if they have a neighbour in $N(C)$, we put them into  $Y$. By Claim~\ref{lem:structure} we know that any vertex not qualifying for $Y$ is dominated by one of its non-neighbours. So we can put such vertices aside in order to colour them  at the very end. This takes $O(m)$ time. 

Now, if $|V(C)|=t-2$, we find sets $Q$ from Claim~\ref{d=3} and $R$ from Claim~\ref{lem:colourW2d+1}, which together form $X$ (in this case). This can be done in $O(n+m)$ time, as we only need to find inclusion minimal neighbourhoods in disjoint sets $D_i$ or $T_i$ (and then pick a vertex for $Q$ or $R$). 

If $|V(C)|=t$, we need to find sets $M$ from Claim~\ref{coroM} and $B$ from Claim~\ref{lem:colourW}, which also takes $O(n+m)$ time. Together, $M$ and $B$ form the set $X'$, from which we obtain the set $X$ of Lemma~\ref{propoX} by adding some extra vertices. More precisely, we add one new vertex each time we find a 
bipartite component $K$  of $G[Y]$, and a vertex $x\in X'$, such that the only neighbours of $x$ in $Y$ belong to $K$, and, at the same time, one of the bipartition classes of $K$ does not send any edge to $X'$. We can explore in $O(n+m)$ time all components $K$ to check if they are as required, and add the extra vertex to $X'$ if needed. 

The total time for finding $C$,  analysing the structure, and finding $X$ is $O(n+m)$.\smallskip

{\bf Checking all precolourings of $X$.}
We need to consider all distinct feasible 
 colourings of $V(C)\cup X$. This set has at most $4t$ vertices, which form a connected set, so we will need to check $3\cdot 2^{4t-1}$ many colourings. Updating can be done in $O(m)$ time. 
 
 We then go through the components of $G[V_3]$ (where $V_3$ are the vertices having list size $3$). Any trivial reducible component can be coloured by checking which colour is missing at any of its neighbours' list. Bipartite reducible components can be dealt with similarly. This takes $O(n)$ time.
 
Finally,  we need to solve a list-colouring
instance with lists of size at most $2$, which can be done in
$O(n+m)$ time. If a colouring is found, we add back to $G$ the vertices dominated by non-neighbours, suitably coloured. 
In conclusion,  if $G\in\mathcal G^*$, then 
the overall
complexity of the algorithm is $2^{O(t)}\cdot O(n+m)$.

{\bf Variation with extra precolouring in case $G\notin\mathcal G^*$.}
We will know we are in this case if $Y^*$ turns out to be non-empty. We need to determine all sets $Y^*_i$ and the corresponding $F(Y^*_i)$. Note that these sets can be found in $O(m)$ time, before we colour the cycle $C$. We then go through all feasible colourings of $C$. There are $3\cdot2^{t-3}$ colourings we need to check. For each of these we colour all $F(Y^*_i)$, and then update, colour the reducible subgraphs and run the $2$-list colouring instance as before. The overall
complexity of the algorithm stays at $2^{O(t)}\cdot O(n+m)$.

%\section*{Acknowledgements}

\bibliographystyle{plain}

\begin{thebibliography}{10}

\bibitem{B-C-M-S-S-Z-colP7}
F.~{Bonomo}, M.~{Chudnovsky}, P.~{Maceli}, O.~{Schaudt}, M.~{Stein}, and
  M.~{Zhong}.
\newblock Three-coloring and list three-coloring of graphs without induced
  paths on seven vertices.
\newblock {\em Combinatorica}, 38(4):779--801, 2018.

\bibitem{mim-width}
N.~Brettell, J.~Horsfield, D.~Paulusma.
\newblock Colouring $(sP_1+P_5)$-Free Graphs: a Mim-Width Perspective.
\newblock Preprint, arXiv:2004.05022 (2020). 

\bibitem{sP3}
H.~Broersma, F.V.~Fomin, P.A.~Golovach, D.~Paulusma.
\newblock Three complexity results on coloring  $P_k$-free graphs. 
\newblock {\it Eur. J. Comb.} 34(3), 609--619, 2013.
 
\bibitem{updating}
H.~Broersma, P.A.~Golovach, D.~Paulusma, J.~Song.
\newblock  Updating the complexity status of coloring graphs without a fixed induced linear forest. 
\newblock {\it Theor. Comput. Sci.} 414(1), 9--19, 2012.


  \bibitem{P6-rP3}
M.~{Chudnovsky}, S.~Huang, S.~{Spirkl}, and M.~{Zhong}.
\newblock List-three-coloring graphs with no induced $P_6+rP_3$.
\newblock {\it Algorithmica}, published online 07/2020.
\newblock https://doi.org/10.1007/s00453-020-00754-y

  \bibitem{col-P6-free-I}
M.~{Chudnovsky}, S.~{Spirkl}, and M.~{Zhong}.
\newblock Four-coloring {$P_6$}-free graphs: Extending an excellent precoloring.
\newblock arXiv:1802.02282 (2018).
  
    \bibitem{col-P6-free-II}
M.~{Chudnovsky}, S.~{Spirkl}, and M.~{Zhong}.
\newblock Four-coloring {$P_6$}-free graphs: Finding an excellent precoloring.
\newblock Preprint arXiv:1802.02283 (2018).

 \bibitem{P8}
M.~Chudnovsky, J.~Stacho.
\newblock 3-Colorable subclasses of $P_8$-free graphs. 
\newblock {\it SIAM J. Discrete Math.}, 32(2), 1111--1138, 2018.
 
    \bibitem{cout}
J.-F.~Couturier, P.A.~Golovach, D.~Kratsch, D.~Paulusma.
\newblock List coloring in the absence of a linear forest. 
\newblock {\em Algorithmica} 71(1), 21--35, 2015.

\bibitem{E-R-T-2-lc}
P.~Erd{\H{o}}s, A.~{Rubin}, and H.~{Taylor}.
\newblock Choosability in graphs.
\newblock {\em Congressus Numerantium}, 26:125--157, 1979.


%\bibitem{GS18}
%J.~Goedgebeur and O.~Schaudt.
%\newblock {Exhaustive generation of $k$-critical $\mathcal H$-free graphs}.
%\newblock {\em Journal of Graph Theory}, 87:188--207, 2018.

%\bibitem{GJ}
%M.R.~Garey and D.S.~Johnson.
%\newblock  Computers and Intractability: A Guide to the Theory of NP-Completeness
%\newblock W. H. Freeman and Co., 1978.

\bibitem{G-J-P-S-col}
P.~A. {Golovach}, M.~{Johnson}, D.~{Paulusma}, and J.~{Song}.
\newblock A survey on the computational complexity of colouring graphs with
  forbidden subgraphs.
\newblock {\em Journal of Graph Theory}, 84(4):331--363, 2017.

%\bibitem{G-P-S-col}
%P.~A. {Golovach}, D.~{Paulusma}, and J.~{Song}.
%\newblock Closing complexity gaps for coloring problems on {$H$}-free graphs.
%\newblock {\em Information and Computation}, 237:204--214, 2014.

\bibitem{G-P-S-path-cycle}
P.A. {Golovach}, D.~{Paulusma}, and J.~{Song}.
\newblock Coloring graphs without short cycles and long induced paths.
\newblock {\em Discrete Applied Mathematics}, 167:107--120, 2014.

\bibitem{G-L-S-perf}
M.~Gr{\"{o}}tschel, L.~Lov{\'a}sz, and A.~{Schrijver}.
\newblock The ellipsoid method and its consequences in combinatorial
  optimization.
\newblock {\em Combinatorica}, 1:169--197, 1981.

\bibitem{Hell-survey-coloring}
P.~{Hell} and S.~{Huang}.
\newblock Complexity of coloring graphs without paths and cycles.
\newblock {\em Discrete Applied Mathematics}, 216(1):211--232, 2017.

\bibitem{H-K-L-S-S-col-P5}
C.T. Ho{\`a}ng, M.~Kami{\'n}ski, V.V. {Lozin}, J~{Sawada}, and X.~{Shu}.
\newblock Deciding $k$-colorability of {$P_5$}-free graphs in polynomial time.
\newblock {\em Algorithmica}, 57:74--81, 2010.

%\bibitem{H-M-R-S-V-obstr-P5}
%C.T. {Ho{\`a}ng}, B.~{Moore}, D.~{Recoskiez}, J.~{Sawada}, and M.~{Vatshelle}.
%\newblock Constructions of {$k$}-critical {$P_5$}-free graphs.
%\newblock {\em Discrete Applied Mathematics}, 182:91--98, 2015.
%
\bibitem{Hol-col}
I. Holyer.
\newblock The NP-completeness of edge-coloring. 
\newblock {\it SIAM J. Comput.} 10(4), 718--720, 1981.

\bibitem{Huang-col-Pt-free}
S.~Huang.
\newblock Improved complexity results on $k$-coloring {$P_t$}-free graphs.
\newblock {\em European Journal of Combinatorics}, 51:336--346, 2016.

%\bibitem{huang2015narrowing}
%S.~Huang, M.~Johnson, and D.~Paulusma.
%\newblock Narrowing the complexity gap for colouring (cs, pt)-free graphs.
%\newblock {\em The Computer Journal}, 58(11):3074--3088, 2015.

\bibitem{K-L-col-g}
M.~Kami{\'n}ski and V.V. {Lozin}.
\newblock Coloring edges and vertices of graphs without short or long cycles.
\newblock {\em Contributions to Discrete Mathematics}, 2:61--66, 2007.

\bibitem{Ka-NPC}
R.~{Karp}.
\newblock Reducibility among combinatorial problems.
\newblock In R.~{Miller} and J.~{Thatcher}, editors, {\em Complexity of
  {C}omputer {C}omputations}, pages 85--103. Plenum Press, New York, 1972.

\bibitem{Paulusma-ISAAC18}
T.~{Klimo\u{s}ov\'a}, J.~Mal{\'{\i}}k, T.~Masar{\'{\i}}k, J.~Novotn{\'a},
  D.~{Paulusma}, and V.~Sl{\'{\i}}vov{\'a}.
\newblock Colouring {$(P_r+P_s)$}-free graphs.
\newblock {\em Algorithmica}, 82, 1833--1858, 2020.

\bibitem{K-K-T-W-col-g}
D.~Kr{\'a}l, J.~Kratochv{\'{\i}}l, Zs. {Tuza}, and G.J. {Woeginger}.
\newblock Complexity of coloring graphs without forbidden induced subgraphs.
\newblock In M.C. {Golumbic}, M.~{Stern}, A.~{Levy}, and G.~{Morgenstern},
  editors, {\em Proceedings of the {I}nternational {W}orkshop on
  {G}raph-{T}heoretic {C}oncepts in {C}omputer {S}cience 2001}, volume 2204 of
  {\em Lecture Notes in Computer Science}, pages 254--262, 2001.

\bibitem{L-G-color}
D.~{Leven} and Z.~{Galil}.
\newblock {NP}-completeness of finding the chromatic index of regular graphs.
\newblock {\em Journal of Algorithms}, 4:35--44, 1983.

%\bibitem{M-M-cert-P5}
%F.~{Maffray} and G.~{Morel}.
%\newblock On 3-colorable {$P_5$}-free graphs.
%\newblock {\em SIAM Journal on Discrete Mathematics}, 26(4):1682--1708, 2012.

\bibitem{M-P}
F.~{Maffray} and M.~{Preissmann}.
\newblock On the {NP}-completeness of the {$k$}-colorability problem for
  triangle-free graphs.
\newblock {\em Discrete Mathematics}, 162:313--317, 1996.

\bibitem{RS}
B.~{Randerath} and I.~{Schiermeyer}.
\newblock 
3-Colorability $\in$ P for $P_6$-free graphs,
\newblock {\em Discrete Applied Mathematics},
136(2-3): 299-313,
2004.

\bibitem{Viz-color}
V.~{Vizing}.
\newblock Coloring the vertices of a graph in prescribed colors.
\newblock {\em Metody Diskretnogo Analiza}, 29:3--10, 1976.

\end{thebibliography}

\end{document}